\documentclass[11pt]{amsart}
\usepackage{amsmath,amsfonts,amssymb,amsthm}
\usepackage[alphabetic]{amsrefs}
\usepackage{hyperref}
\usepackage{tikz}
\usepackage{graphicx,color}
\usepackage{pictexwd,dcpic,epsf}
\usepackage{enumerate}

\newtheorem{theorem}{Theorem}[section]
\newtheorem{lemma}[theorem]{Lemma}
\newtheorem{prop}[theorem]{Proposition}

\newtheorem{corollary}[theorem]{Corollary}

\newtheorem*{maitheorem}{Main Theorem}

\newtheorem*{cor}{Corollary}

\theoremstyle{definition} 
\newtheorem{definition}[theorem]{Definition}
 
\newtheorem{example}[theorem]{Example}

\theoremstyle{remark} 
\newtheorem{remark}[theorem]{Remark}
\usepackage{enumitem}

\newcommand{\ov}{\overline}

\newcommand{\angled}[1]{\langle#1\rangle}

\newcommand {\mr}{\mathrm}

\newcommand{\Xa}{X^{\ast}}
\newcommand{\Pa}{\Pi^{\ast}}
\newcommand{\Xb}{X^{\scalebox{0.53}{$\triangle$}}}

\newcommand{\ui}[1]{u_{#1}^{-1}}
\newcommand{\di}[1]{d_{#1}^{-1}}

\newcommand{\prism}[2]{\mr{Prism}(#1,#2)}

\begin{document}

\title[
Large-type Artin groups are systolic
]{
Large-type Artin groups are systolic
}

\author{Jingyin Huang}
\address{Dept.\ of Math.\ \& Stats., McGill University\\ Montreal, Quebec, Canada H3A 0B9}
\email{jingyin.huang@mcgill.ca}

\author{Damian Osajda}
\address{Instytut Matematyczny,
Uniwersytet Wroc\l awski\\
pl.\ Grun\-wal\-dzki 2/4,
50--384 Wroc{\l}aw, Poland}
\address{Institute of Mathematics, Polish Academy of Sciences\\
\'Sniadeckich 8, 00-656 War\-sza\-wa, Poland}
\address{Dept.\ of Math.\ \& Stats., McGill University\\ Montreal, Quebec, Canada H3A 0B9}
\email{dosaj@math.uni.wroc.pl}


\date{\today}

\begin{abstract}
We prove that Artin groups from a class containing all large-type Artin groups are systolic. This provides a concise yet precise description of their geometry.
Immediate consequences are new results 
concerning large-type Artin groups: biautomaticity; existence of $EZ$-boundaries; the Novikov conjecture;
descriptions of finitely presented subgroups, of virtually solvable subgroups, and of centralizers of elements; the Burghelea conjecture; existence of low-dimensional models for classifying spaces for some
families of subgroups.
\end{abstract}

\maketitle
\tableofcontents
\setcounter{tocdepth}{2}

\section{Introduction}
\label{s:intro}

%

\subsection{Background and the Main Theorem}
Let $\Gamma$ be a finite simple graph with its vertex set denoted by $V$. Let each edge of $\Gamma$ be labeled by a positive integer at least two. The \emph{Artin group with defining graph $\Gamma$}, denoted $A_{\Gamma}$, is the group whose generating set is $V$, and whose relators are of the form $\underbrace{aba\cdots}_{m}=\underbrace{bab\cdots}_{m}$ for each $a$ and $b$ spanning an edge labeled by $m$. 

It is an open question whether all Artin groups are non-positively curved in the sense that they act geometrically (i.e.\ properly and cocompactly by isometries) on non-positively curved spaces. One of the earlier motivations for this question comes from the seminal work of Charney and Davis \cite{charney1995k}, where they put a $CAT(0)$ metric on the modified Deligne complex for Artin groups of type FC as well as for $2$-dimensional Artin groups, and deduce the $K(\pi,1)$ conjecture for these Artin groups. Though the action of an Artin group on its modified Deligne complex is not proper, this naturally leads to the question of whether one can directly construct $CAT(0)$ spaces on which Artin groups act geometrically. This question is of independent interest to the $K(\pi,1)$ conjecture, since one can deduce many finer group theoretic and geometric consequences given the existence of such action. Here is a summary of Artin groups which are known to act geometrically on $CAT(0)$ spaces:
\begin{enumerate}
	\item right-angled Artin groups \cite{ChaDav1995};
	\item certain classes of $2$-dimensional Artin groups \cite{brady2002two,BradyMcCammond2000};
	\item Artin groups of finite type with three generators \cite{brady2000artin};
	\item $3$-dimensional Artin groups of type FC \cite{bell2005three};
	\item spherical Artin groups of type $A_4$ and $B_4$ \cite{brady2010braids};
	\item the $6$-strand braid group \cite{haettel20166}.
\end{enumerate}

In this paper we focus on Artin groups of \emph{large type}, i.e.\ those whose defining graphs have edge labels of value at least three. Large-type Artin groups were first introduced and studied by Appel and Schupp \cite{AppelSchupp1983}.
It is still unknown whether all Artin groups of large type are $CAT(0)$, though some partial results were obtained in \cite{BradyMcCammond2000}. Moreover, most Artin groups of large type can not act geometrically on $CAT(0)$ cube complexes, even up to passing to finite index subgroups \cite{huang2015cocompactly,haettel2015cocompactly}.
\medskip

Instead of metric non-positive curvature, we turn our attention towards a combinatorial 
counterpart. Examples of combinatorially non-positively curved spaces and groups are: small cancellation groups, 
CAT(0) cubical groups, and systolic groups. There are many advantages to the combinatorial approach. For example, 
biautomaticity has been proved in various combinatorial settings, while it is still an open problem (with a plausible
negative answer) for CAT(0) groups (see the discussion in Subsection~\ref{s:applications} below). 

A suitable setting for our approach are Artin groups in the following class.
An Artin group is of \emph{almost large type} if in the defining graph $\Gamma$ there is no triangle with an edge labeled by two and no square with three edges labeled by two. Clearly, large-type
Artin groups are of almost large type. Right-angled Artin groups with their defining graphs being triangle free and square free are examples
of almost large-type Artin groups that are not of large type. Our main result is the following (see Theorem~\ref{thm:main} in Section~\ref{s:general}).
\begin{maitheorem}
	Every Artin group of almost large type is systolic. 
\end{maitheorem}
\emph{Systolic groups} are groups acting geometrically on \emph{systolic complexes}. The latter are simply connected simplicial complexes
satisfying some local combinatorial conditions implying many non-positive-curvature-like features (see Subsection~\ref{s:applications}, and Section~\ref{s:syst} below for some details).  
Systolic complexes were first introduced by Chepoi \cite{Chepoi2000} under the name \emph{bridged complexes}.
However, \emph{bridged graphs}, one-skeleta of systolic complexes, were studied earlier in the context of metric graph theory.
They were introduced by Soltan-Chepoi \cite{SoltanChepoi1983} and Farber-Jamison \cite{FarberJamison1987}.
Systolic complexes were rediscovered independently by Januszkiewicz-\'Swi{\c a}tkowski \cite{JanuszkiewiczSwiatkowski2006} and by Haglund \cite{Haglund} in the context of geometric group theory.
The combinatorial approach to non-positive curvature allowed for the construction of groups and complexes
with interesting properties. In particular, the first examples discovered of high-dimensional
hyperbolic Coxeter groups were systolic. The theory of systolic complexes and groups has been developed extensively
providing new applications (see e.g.\ \cite{JanuszkiewiczSwiatkowski2007,Wise2003-sixtolic,Swiatkowski2006,OsajdaPrytula} and references therein).  

Let us note that there have been other very successful approaches to Artin groups using combinatorial versions
of non-positive curvature. Those include: using small cancellation \cite{AppelSchupp1983,Pride,Peifer},
CAT(0) cube complexes (the case of right-angled Artin groups), and Bestvina's approach to
Artin groups of finite type \cite{Bestvina1999}.

To prove our main theorem, we construct a systolic complex on which the Artin group acts. This complex is a thickening of the presentation complex of the Artin group. Disk diagrams in systolic complexes are very simple (\cite[Lemma 4.2]{Elsner2009-flats}), so they can used to study disk diagrams in the presentation complexes of these Artin groups through our systolic thickening. Hence we believe that our complexes can be used to prove finer properties of Artin groups of almost large type, beyond those presented in Subsection~\ref{s:applications}. Also we expect that our approach can be adapted for more general classes of Artin groups.

\subsection{Comments on the proof}
First we consider the special case of $A_\Gamma$ where the label of each edge in $\Gamma$ is three. Let $X^{\ast}_{\Gamma}$ be the universal cover of the presentation complex of $A_{\Gamma}$. Then each $2$-cell of $X^{\ast}_{\Gamma}$ is a hexagon. We put a new vertex (called an \emph{interior vertex}) in the interior of each $2$-cell and subdivide each $2$-cell into $6$ triangles around this new vertex. One naturally wants to metrize such a complex by declaring each triangle to be a Euclidean equilateral triangle. However, such metric is not $CAT(0)$, since there exist pairs of $2$-cells intersecting along two edges, and this leads to positively curved points (as vertex $v$ in Figure~\ref{fig:short}). We think of the configuration around $v$ as a ``corner" inside a $3$-dimensional Euclidean space that we would like to ``fill in" to kill the positive curvature. Specifically, we add an edge between the interior vertices of every pair of $2$-cells whose intersection contains $\geqslant 2$ edges, and take the flag complex. Though the new complex is still not $CAT(0)$, it appears to have enough non-positive curvature properties to work with and the suitable language to realize this intuition is the theory of systolic complexes.
\begin{figure}[h!]
	\centering
	\includegraphics[scale=0.6]{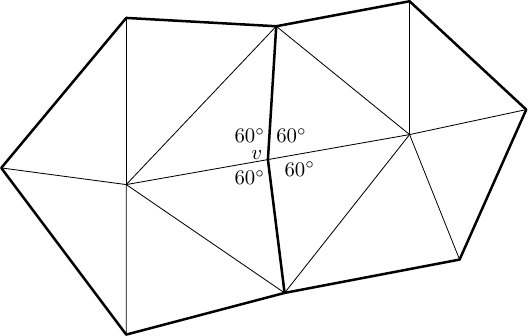}
	\caption{The vertex $v$ is positively curved.}
	\label{fig:short}
\end{figure}

Next we consider the more general case of $A_\Gamma$ of large type. Again we start with the universal cover of the presentation complex and subdivide each $2$-cell, now a $2n$-gon for some $n\geqslant 3$, into triangles. Among the many possibilities of subdivision, we choose the one as in Figure~\ref{f:cell} (more than one interior vertices are added when $n>3$) based on the following considerations:
\begin{enumerate}
	\item Each triangle is a Euclidean equilateral triangle.
	\item Each $2$-cell with the subdivision is flat.
\end{enumerate}
The reason for (2) is that $A_\Gamma$ contains many rank two free abelian subgroups, so intentionally creating negative curvature at a point forces there to be positive curvature at some other points.

As in the previous case, a pair of $2$-cells with a large piece between them lead to points of positive curvature. We add edges between the interior vertices of these $2$-cells to create a \textquotedblleft prism-like\textquotedblright\ configuration as in Figure~\ref{f:cell2}, which resolves these positive curvature points. The general idea of adding edges in order to ``systolize" some complexes has been used before \cite{PrzytyckiSchwer}.

The bulk of this paper (Section~\ref{sec:link of vertices}) will be devoted to the study of the above complexes for the dihedral Artin groups (i.e. the $A_\Gamma$ with $\Gamma$ being an edge), since these complexes are our building blocks in the study of more general Artin groups. We show that these building blocks are systolic (Proposition~\ref{prop:real vertex link is six-large}, Proposition~\ref{prop:interior vertex link is six-large}). Moreover, the \textquotedblleft prism-like\textquotedblright\ configuration in the previous paragraph needs to be designed carefully so that there is no obstruction to systolicity if we glue these blocks together (Lemma~\ref{l:ain far from aout}, Lemma~\ref{lem:distance}). The procedure of gluing the building blocks together is explained in Section~\ref{s:general}. The complexes for almost large type Artin groups are defined in Definition~\ref{def:construct XGa}.

We end this subsection by noting that dihedral Artin groups are already very well-understood: they are virtually free times $\mathbb Z$, they are known to act on various complexes with features of non-positive curvature \cite{Bestvina1999,BradyMcCammond2000,brady2002two,huang2015cocompactly,haettel2015cocompactly}, and it is known that one can use them as building blocks to obtain $CAT(0)$ complexes for a certain family of Artin groups. However, these building blocks are not good enough for constructing $CAT(0)$ complexes for all Artin groups of large type (at the time of writing this paper, even the case where the defining graph $\Gamma$ is a complete graph on more than three vertices is not known).


\subsection{Immediate consequences of the Main Theorem}
\label{s:applications}
We gather immediate consequences of systolicity for almost large-type Artin groups in the following corollary.
To the best of our knowledge all the results listed here are new.
 Below we provide some details, in particular, we comment on earlier results.
\begin{cor}
	\label{c:appl}
	Let $G$ be an Artin group of almost large type. Then:
	\begin{enumerate}
		\item $G$ is biautomatic;
		\item $G$ has a boundary in the sense of \cite{OsajdaPrzytycki}, which captures the large-scale geometry of $G$. In particular, $G$ admits an EZ--structure, and hence the Novikov conjecture holds for $G$;
		\item finitely presented subgroups of $G$ are systolic, hence they are biautomatic, have solvable word problem, solvable conjugacy problem and all the other properties listed here;
		\item virtually solvable subgroups of $G$ are either virtually cyclic or virtually $\mathbb Z^2$;
		\item the Burghelea conjecture holds for $G$;
		\item the centralizer of an infinite order element of $G$ is commensurable with $F_n \times \mathbb Z$ or $\mathbb Z$;
		\item $G$ admits a finitely dimensional model for $E_{\mathcal{VAB}}G$,  the classifying space for
		the family of virtually abelian subgroups of $G$.
	\end{enumerate}
\end{cor}

\noindent (1)
Biautomaticity for systolic groups has been established by Januszkiewicz-\' Swi\c atkowski \cite{JanuszkiewiczSwiatkowski2006,Swiatkowski2006}. Biautomaticity of large-type Artin groups was a well known open problem.
Partial results 
were obtained by: Pride together with Gersten and Short (triangle-free Artin groups) \cite{Pride,Gersten}, Charney \cite{Charney1992} (finite type), Peifer \cite{Peifer} (extra-large type, i.e., $m_{ij}\geqslant 4$, for $i\neq j$), Brady-McCammond \cite{BradyMcCammond2000} (three-generator large-type Artin groups and some generalizations), Holt-Rees \cite{HoltRees2012,HR2013} (sufficiently large Artin groups are shortlex automatic with respect to the standard generating set).

Biautomaticity has many important consequences. Among them are: quadratic Dehn function, solvability of
the Word Problem, and of the Conjugacy Problem. 
Chermak \cite{Chermak} proved that the Word Problem is solvable for $2$-dimensional Artin groups, hence for
all Artin groups of almost large type. The Conjugacy Problem was known to be solvable for large-type Artin groups
by results of Appel-Schupp \cite{AppelSchupp1983} and Appel \cite{Appel1984}, but there have been no  results
about other Artin groups of almost large type in general.
It follows from a work of Holt-Rees \cite{HR2013} that the Dehn function is quadratic for sufficiently large
Artin groups. All almost large-type Artin groups are sufficiently large. 
To the best of our knowledge there have been no general results concerning the above problems for
finitely presented subgroups of Artin groups in question. As explained below in (3) our results
apply to them as well.

\medskip

\noindent
(2) Let $G$ act geometrically on a systolic complex $X$. Osajda and Przytycki \cite{OsajdaPrzytycki} constructed a compactification
of $X$ by a $Z$-set $\ov X=X\cup \partial X$. This defines the so-called $EZ$-structure
for $G$ \cite{FarrellLafont2005}, and $\partial X$ becomes a sort of a boundary of $G$. Such structures
are known only for a few classes of groups, most notably, for Gromov hyperbolic groups and CAT(0) groups.
Closer relations between algebraic properties of $G$ and the dynamics of its action on $\partial X$ are
exhibited in \cite{Prytula2017}.
Existence of an $EZ$-structure implies, in particular, the Novikov conjecture \cite{FarrellLafont2005}.  
Ciobanu-Holt-Rees \cite{CiobanuHoltRees2016} established the (stronger) Baum-Connes conjecture for a subclass of 
large-type Artin groups, including Artin groups of extra-large type. They did it by proving the rapid decay property for such groups. 

\medskip

\noindent
(3) Using towers of complexes Wise \cite{Wise2003-sixtolic} showed that finitely presented subgroups
of torsion-free systolic groups are systolic. For all systolic groups the result has been shown in \cite{HanlonMartinez,Zadnik}.

\medskip

\noindent
(4) By Theorem~\ref{t:virtsolv} in Section~\ref{s:syst} below,  virtually solvable subgroups of systolic groups
are either virtually cyclic or virtually $\mathbb Z^2$. Bestvina \cite{Bestvina1999} showed that solvable
subgroups of Artin groups of finite type are abelian. He used another combinatorial version of
non-positive curvature.
It is an open question whether virtually solvable subgroups of biautomatic groups are finitely generated virtually abelian groups. The result is known for polycyclic subgroups of biautomatic groups (and so for biautomatic groups all of whose abelian subgroups are finitely generated) by work of Gersten-Short \cite{GerstenShort1991}. 
\medskip

\noindent
(5) The Burghelea conjecture concerns the periodic cyclic homology of complex group rings; see \cite{EngelMarcinkowski2016} and references therein for details. It is known to be false in general,
but has been established for, among others, hyperbolic groups. Engel-Marcinkowski \cite{EngelMarcinkowski2016} showed that the Burghelea conjecture holds for systolic groups.
The Burghelea conjecture implies the strong Bass conjecture that, in turn implies
the classical Bass conjecture.  

\medskip

\noindent
(6) Elsner \cite{Elsner2009-isometries} showed that infinite order elements of systolic groups admit
a kind of axis, similarly to the hyperbolic and CAT(0) cases. Verifying a conjecture by Wise \cite{Wise2003-sixtolic}, it is shown in \cite{OsajdaPrytula} that centralizers of infinite order elements
in systolic groups are commensurable with a product of $\mathbb Z$ and a finitely generated free
 group (possibly trivial or $\mathbb Z$). In the special case of $2$-dimensional Artin groups of hyperbolic type, this result was obtained by Crisp \cite{MR2174269}. In fact, Crisp computes explicitly the centralizer of a given element up to commensurability.

\medskip

\noindent
(7) By a result of Degrijse \cite{Degrijse2017} two-dimensional Artin groups admit finite dimensional
models for the family of virtually cyclic subgroups. A similar result has been proved in \cite{OsajdaPrytula}
for systolic groups, where it is also shown that systolic groups admit finite dimensional models
for classifying spaces for the family of virtually abelian subgroups.

\medskip

\noindent
{\bf Acknowledgments.} 
We thank Tomasz Prytu\l a for pointing out the proof of the Solvable Subgroup Theorem.
We thank Nima Hoda, Piotr Przytycki, and the anonymous referee for useful remarks.
The authors were partially supported by (Polish) Narodowe Centrum Nauki, grant no.\ UMO-2015/\-18/\-M/\-ST1/\-00050. The paper was written while D.O.\ was visiting McGill University.
We would like to thank the Department of Mathematics and Statistics of McGill University
for its hospitality during that stay.

\section{Systolic complexes and systolic groups}
\label{s:syst}
All graphs considered in this paper neither contain edge-loops nor multiple edges. For vertices $w,v_1,v_2,\ldots$
of a graph, we write $w\sim v_1,v_2,\ldots$ when $w$ is \emph{adjacent} to each $v_i$, that is, there is an edge containing $w$ and $v_i$. If $w$ is not adjacent to any of $v_i$ then we write $w\nsim v_1,v_2,\ldots$.
A simplicial complex $X$ is \emph{flag} if every set of pairwise adjacent vertices of $X$ spans a simplex
in $X$. In other words, a flag simplicial complex $X$ is determined by its \emph{$1$-skeleton} $X^{(1)}$,
being a simplicial graph. 
A subcomplex $Y$ of a simplicial complex $X$ is \emph{full} if any set of vertices of $Y$ spanning a simplex
in $X$ spans a simplex in $Y$ as well.
A subgraph of a graph is \emph{full} if it is a full subcomplex. 
The \emph{link} lk$(v,\Gamma)$ of a vertex $v$ in a graph $\Gamma$ is the full subgraph of $\Gamma$ spanned by vertices
adjacent to $v$.
A graph is \emph{$6$-large} if there are no simple cycles of length $4$ and $5$ being full 
subgraphs.

\begin{definition}
	\label{d:systolic}
	A flag simplicial complex is \emph{systolic} if it is connected, simply connected, and links
	of vertices in its $1$-skeleton are $6$-large.
\end{definition}

In particular, any $2$-dimensional piecewise Euclidean $CAT(0)$ complex whose $2$-cells are equilateral triangles satisfies the above definition. In general, systolic complexes are not $2$-dimensional and they are not necessarily $CAT(0)$ with the most natural metric -- the piecewise Euclidean metric with all edges having
length $1$. Nevertheless, systolic complexes possess many features typical for nonpositively curved spaces.
One of them is a version of the Cartan-Hadamard theorem stating that finite dimensional systolic complexes are contractible. Another important feature is that for any embedded simplicial loop in a systolic complex
there is a systolic disc diagram filling it. Such a diagram can be equipped with a CAT(0) structure. See e.g.\ \cite{Chepoi2000,Haglund,Wise2003-sixtolic,JanuszkiewiczSwiatkowski2006,JanuszkiewiczSwiatkowski2007,Elsner2009-flats,Elsner2009-isometries,OsajdaPrzytycki,Zadnik,OsajdaPrytula} for details and further information.

Groups acting geometrically on systolic complexes are called \emph{systolic}.
Few consequences of being a systolic group are listed in Corollary above.
The following theorem is a consequence of known results on systolic complexes but has been
not stated in the literature.\footnote{It is mistakenly claimed in \cite{JanuszkiewiczSwiatkowski2006,JanuszkiewiczSwiatkowski2007} that a form of Solvable Subgroup Theorem follows immediately from biautomaticity. In fact it is an open question whether virtually solvable subgroups of biautomatic groups are virtually abelian.}
\begin{theorem}[Solvable Subgroup Theorem]
	\label{t:virtsolv}
	Solvable subgroups of systolic groups are either virtually cyclic or virtually $\mathbb{Z}^2$. 
\end{theorem}
\begin{proof}
Let $G$ be a systolic group.
By \cite[Proposition 5.10]{OsajdaPrytula} virtually abelian subgroups of $G$ are finitely generated.
Hence, the following argument by Gersten and Short \cite[page 154]{GerstenShort1991} shows that
solvable subgroups of $G$ are virtually abelian:
By a theorem of Mal'cev \cite[Theorem 2 on page 25]{Segal1083} such subgroups are polycyclic, and thus, by \cite[Theorem 6.15]{GerstenShort1991}
they are virtually abelian.
	By 
	\cite[Corollary 6.5]{JanuszkiewiczSwiatkowski2007} 
	virtually abelian subgroups of $G$ have rank at most $2$. 
\end{proof}

\section{The complexes for $2$-generated groups}
\label{s:dihedral}

\subsection{Precells in the presentation complex}
\label{subsec:precells}
Let $DA_n$ be the $2$-generator Artin group presented by $\angled{a,b\mid \underbrace{aba\cdots}_{n} =
	\underbrace{bab\cdots}_n}$. We assume $n\geqslant 2$. We define $DA^{+}_n$ to be the associated Artin monoid presented by the same generators and relations as $DA_n$. 
\begin{lemma}
	\label{lem:words in the positive monoid}
Let $w_1$ and $w_2$ be two words in the free monoid generated by $a$ and $b$. Suppose $w_1=w_2$ in $DA^{+}_n$. Then
\begin{enumerate}
	\item $w_1$ and $w_2$ have the same length;
	\item if $w_1$ has length $\leqslant n$, then either $w_1$ and $w_2$ are the same word, or $w_1$ equals to one of $\underbrace{aba\cdots}_{n}$ and $\underbrace{bab\cdots}_n$, and $w_2$ equals to another.
\end{enumerate}
\end{lemma}

\begin{proof}
Note that $w_1=w_2$ implies that one can obtain $w_2$ from $w_1$ by applying the relation finitely many times. But applying the relation does not change the length of the word, thus (1) follows. For (2), if $w_1$ has length $\le n$ and $w_1$ is not equal to one of the words appearing in the relation, then there is no way to apply the relation to transform $w_1$ to a different word, thus $w_1$ and $w_2$ have to be the same word.
\end{proof}

The following is a special case of \cite[Theorem 4.14]{Deligne}.
\begin{theorem}
	\label{thm:injective artin monoid}
The natural map $DA^{+}_n\to DA_n$ is injective.	
\end{theorem}

Let $P_n$ be the standard presentation complex of $DA_n$. Namely the $1$-skeleton of $P_n$ is the wedge of two oriented circles, one labeled $a$ and one labeled $b$. Then we attach the boundary of a closed $2$-cell $C$ to the $1$-skeleton with respect to the relator of $DA_n$. Let $C\to P_n$ be the attaching map.

Let $\Xa$ be the universal cover of the standard presentation complex of $DA_n$. Edges of $\Xa$ are endowed with induced orientations and labellings from $P_n$. The following is a direct consequence of Theorem \ref{thm:injective artin monoid} and Lemma \ref{lem:words in the positive monoid}.

\begin{corollary}
	\label{cor:cell embeded}
Any lift of the map $C\to P_n$ to $C\to \Xa$ is an embedding.
\end{corollary}

These embedded disks in $\Xa$ are called \emph{precells}. The following is a picture of a precell $\Pa$. Note that $\Xa$ is a union of copies of $\Pa$'s.

\begin{figure}[h!]
	\centering
	\includegraphics[width=1\textwidth]{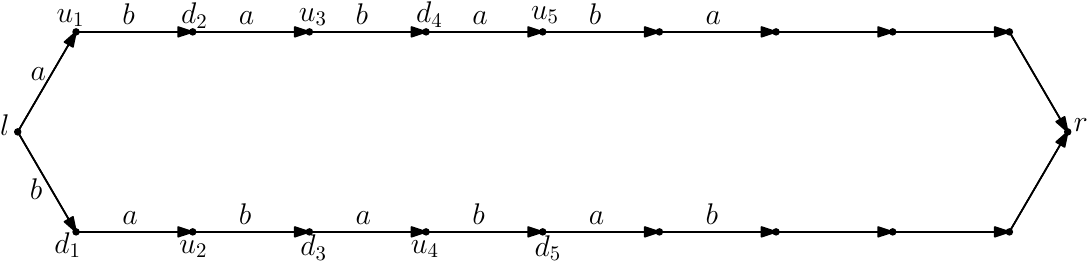}
	\caption{Precell $\Pa$}
	\label{f:precell}
\end{figure}

We label the vertices of $\Pa$ as in Figure \ref{f:precell}. More precisely, the left most vertex and right most vertex are labeled by $l$ and $r$, and they are called the \emph{left tip} and \emph{right tip} of $\Pa$. The boundary $\partial\Pa$ is made of two paths. The one starting at $l$, going along $\underbrace{aba\cdots}_{n}$ (resp.\ $\underbrace{bab\cdots}_{n}$), and ending at $r$ is called the \emph{upper half} (resp.\ \emph{lower half}) of $\partial\Pa$. Vertices in the interior of the upper half are labeled $u_1,d_2,u_3,d_4,\ldots$ from left to right. Vertices in the interior of the lower half are labeled $d_1,u_2,d_3,u_4,\ldots$ from left to right. The orientation of edges inside one half is consistent, thus each half has an orientation. Observe that vertices with labels $u_i$ (resp.\ $d_i$) are terminal (resp.\ initial)
vertices of edges labeled by $a$, and initial (resp.\ terminal) vertices of edges labeled by $b$.
\begin{corollary}
	\label{cor:connected intersection}
Let $\Pa_1$ and $\Pa_2$ be two different precells in $\Xa$. Then 
\begin{enumerate}
	\item either $\Pa_1\cap\Pa_2=\emptyset$, or $\Pa_1\cap\Pa_2$ is connected;
	\item if $\Pa_1\cap\Pa_2\neq\emptyset$ then $\Pa_1\cap\Pa_2$ is properly contained in the upper half or in the lower half of $\Pa_1$ (and of $\Pa_2$);
	\item if $\Pa_1\cap\Pa_2$ contains at least one edge, then one end point of $\Pa_1\cap\Pa_2$ is a tip of $\Pa_1$, and another end point of $\Pa_1\cap\Pa_2$ is a tip of $\Pa_2$, moreover, among these two tips, one is a left tip and one is a right tip.
\end{enumerate}
\end{corollary}

\begin{proof}
First we look at the case when $\Pa_1\cap\Pa_2$ is discrete. Suppose by contradiction that there are two distinct vertices $v_1,v_2$ in $\Pa_1\cap\Pa_2$. 

If $v_1$ and $v_2$ are in the same half of $\partial\Pa_1$ and $\partial\Pa_2$ then, for $i=1,2$, let $p_i$ be the segment joining $v_1$ and $v_2$ inside a half of $\partial\Pa_i$. If both $p_1$ and $p_2$ are oriented from $v_1$ to $v_2$, then they give two words in the free monoid that are equal in $DA_n$. By Theorem~\ref{thm:injective artin monoid} and Lemma~\ref{lem:words in the positive monoid}, these two words have to be in the two situations indicated in Lemma~\ref{lem:words in the positive monoid} (2), however, both situations can be ruled out easily. If $p_1$ is oriented from $v_1$ to $v_2$ and $p_2$ is oriented from $v_2$ and $v_1$, then the concatenation of $p_1$ and $p_2$ gives a nontrivial word in the free monoid, which is also nontrivial in $DA_n$ by Theorem \ref{thm:injective artin monoid}. This contradicts the fact that the concatenation is a loop. Other cases of orientations of $p_1$ and $p_2$ can be dealt in a similar way.

If $v_1$ and $v_2$ are in different halves of $\partial\Pa_1$ and $\partial\Pa_2$ then we assume without loss of generality that orientations of halves of $\Pa_1$ and $\Pa_2$ are as in Figure~\ref{fig:intersect} ($l_i$ and $r_i$ are the tips of $\partial\Pi_i$). We also assume without loss of generality that the summation of the length of the path $\overline{l_2v_1r_1}$ and the path $\overline{l_2v_2r_1}$ is $\leqslant 2n$. Let $w_1$ (resp.\ $w_2$) be the word in the free monoid given by $\overline{l_2v_1r_1}$ (resp.\ $\overline{l_2v_2r_1}$). Then at least one of $w_1$ and $w_2$ has length $\leqslant n$. Again, $w_1$ and $w_2$ are in the two situations of Lemma \ref{lem:words in the positive monoid} (2), and both situations can be ruled out easily.

\begin{figure}[ht!]
	\centering
	\includegraphics[scale=0.75]{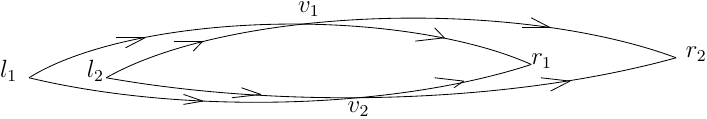}
	\caption{}
	\label{fig:intersect}
\end{figure}

The case where $v_1$ and $v_2$ are in different halves of one of $\partial\Pa_1$ and $\partial\Pa_2$, and are in the same half of the other can be handled in a similar way.

Now we assume $\Pa_1\cap\Pa_2$ contains an edge $f$. Let $P$ be the connected component of $\Pa_1\cap\Pa_2$ that contains this edge. By looking at the labels of edges around $\partial\Pa_1$ and $\partial\Pa_2$, we deduce that either $P=\partial\Pa_1=\partial\Pa_2$, or $P$ satisfies conditions (2) and (3) in Corollary~\ref{cor:connected intersection}. However, the first case is impossible since that will imply $\Pa_1=\Pa_2$. Let $C$ be the space obtained by gluing $\Pa_1$ and $\Pa_2$ along $P$. Then there is a natural map $C\to X^{\ast}$. It suffices to show this map is an embedding. We assume without loss of generality that $\Pa_1$ and $\Pa_2$ are positioned as in Figure \ref{fig:intersect1}, here $l_i$ and $r_i$ are the tips of $\Pa_i$, and $w_1$ (resp.\ $w_2$) is an interior vertex of the upper half of $\Pa_1$ (resp.\ lower half of $\Pa_2$).
\begin{figure}[h!]
	\centering
	\includegraphics[scale=0.75]{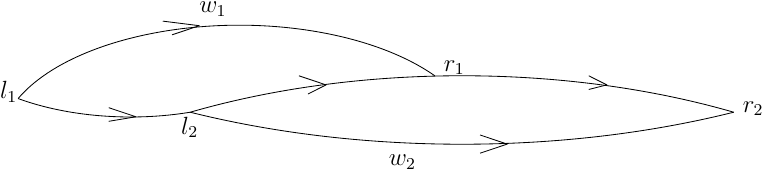}
	\caption{}
	\label{fig:intersect1}
\end{figure}

We already know that $\partial\Pa_1$ and $\partial\Pa_2$ are embedded by Corollary \ref{cor:cell embeded}, and the paths $\overline{l_1l_2r_1r_2}$, $\overline{l_1w_1r_1r_2}$ and $\overline{l_1l_2w_2r_2}$ are embedded because they correspond to words in the free monoid. If $w_1$ and $w_2$ are identified in $X^{\ast}$, then $\overline{l_1w_1}$ and $\overline{l_1l_2w_2}$ give two words in the free monoid which are equal in $DA_n$. By Theorem \ref{thm:injective artin monoid} and Lemma \ref{lem:words in the positive monoid} (2), these two words are identical (note that the length of $\overline{l_1w_1}$ is $< n$), which is a contradiction. 
\end{proof}

\begin{corollary}
	\label{cor:disjoint}
Suppose there are three precells $\Pa_1$, $\Pa_2$ and $\Pa_3$ such that $\Pa_1\cap \Pa_2$ is a nontrivial path $P_1$ in the upper half of $\Pa_2$, and $\Pa_3\cap \Pa_2$ is a nontrivial path $P_3$ in the lower half of $\Pa_2$. Then $\Pa_1\cap \Pa_3$ is either empty or one point.
\end{corollary}

\begin{proof}
We glue $\Pa_2$ and $\Pa_1$ along $P_1$, and glue $\Pa_2$ and $\Pa_3$ along $P_3$ to obtain a space $C$. There is a natural map $C\to X^{\ast}$. By Corollary \ref{cor:cell embeded} (2) and (3), there are four possibilities of the space $C$, we only consider the two cases in Figure \ref{fig:intersect2}, the other cases are similar. Let $l_i$ and $r_i$ be the left tip and right tip of $\Pa_i$.

\begin{figure}[h!]
	\centering
	\includegraphics[scale=0.75]{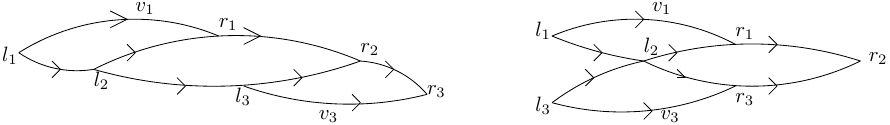}
	\caption{}
	\label{fig:intersect2}
\end{figure}

Suppose we are in the case as in Figure \ref{fig:intersect2}, on the left (namely the case where $l_2\in\Pa_1$ and $r_2\in\Pa_3$). We claim $\partial\Pa_1\cap\partial\Pa_3=\emptyset$. Since $\Pa_2\cup_{P_3}\Pa_3$ is embedded in $X^{\ast}$ by Corollary \ref{cor:cell embeded}, the path $\overline{l_2r_1}$ is disjoint from $\partial\Pa_3$. Similarly, $\overline{l_3r_2}$ is disjoint from $\partial\Pa_1$. Moreover, $\overline{l_1l_2}$ is disjoint from $\overline{r_2r_3}$ and $\overline{l_3v_3r_3}$, since $\overline{l_1l_2l_3v_3r_3}$ and $\overline{l_1l_2r_1r_2r_3}$ give words in the free monoid. 
Similarly $\overline{r_2r_3}$ is disjoint from $\overline{l_1v_1r_1}$.
It remains to show $\overline{l_1v_1r_1}\cap\overline{l_3v_3r_3}=\emptyset$. If this is not true, we assume without loss of generality that $v_1$ and $v_3$ are identified. Then $\overline{l_1v_1}$ and $\overline{l_1l_2l_3v_3}$ give two words in the free monoid which are equal in $DA_n$. Since $\overline{l_1v_1}$ has length $<n$, these words are identical by Theorem \ref{thm:injective artin monoid} and Lemma \ref{lem:words in the positive monoid}, which yields a contradiction. 

Suppose we are in the case of Figure \ref{fig:intersect2} right (namely $l_2\in\Pa_1\cap\Pa_3$). By looking at the labels of edges around vertex $l_2$, we know $l_2$ is an isolated vertex in $\partial\Pa_1\cap\partial\Pa_3$. Thus $\partial\Pa_1\cap\partial\Pa_3=l_2$ by Corollary \ref{cor:cell embeded}.
\end{proof}

\subsection{Subdividing and systolizing the presentation complex}
\label{subsec:subividing and adding new edges}
We subdivide each precell in $X^{\ast}$ as in Figure \ref{f:cell} to obtain a simplicial complex $\Xb$. 
In particular, in the case $n=2$ we do not add any new vertices only an edge $\ov{lr}$. 
A \emph{cell} of $\Xb$ is defined to be a subdivided precell, and we use the symbol $\Pi$ for denoting a cell. 
The original vertices of $X^{\ast}$ in $\Xb$ are called the \emph{real vertices}, and the new vertices of $\Xb$ after subdivision are called \emph{interior vertices}. Interior vertices in a cell $\Pi$ are denoted $c_1,c_2,\ldots,c_{n-2}$ as in Figure \ref{f:cell}. (Here and further we use the convention that the real vertices are drawn as solid points and the interior vertices as circles.)

\begin{figure}[ht!]
	\centering
	\includegraphics[width=1\textwidth]{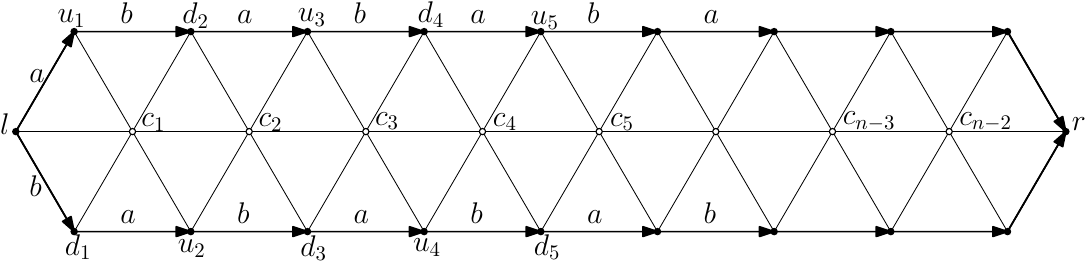}
	\caption{A cell $\Pi$ subdivided into several smaller triangles.}
	\label{f:cell}
\end{figure}

If $n=2$ then $\Xb$ is systolic---it is isomorphic to the equilateral triangulation of the Euclidean plane (see Remark~\ref{rem:n2} below)---and we define $X$ to be $\Xb$. Form now on we assume $n\geqslant 3$. Note that then $\Xb$ is not systolic. Suppose $\Pi_1$ and $\Pi_2$ are two cells such that $\Pi_1\cap\Pi_2$ is a path made of $\geqslant 2$ edges. Then they create $4$-cycles or $5$-cycles in $\Xb$ without diagonals, see the thick cycles in Figure \ref{f:cell2}. In what follows we modify $\Xb$ to obtain a systolic complex $X$. 
A rough idea is to add appropriate diagonals to these $4$-cycles or $5$-cycles. We only add new edges between interior vertices. 

\begin{figure}[h!]
	\centering
	\includegraphics[width=1\textwidth]{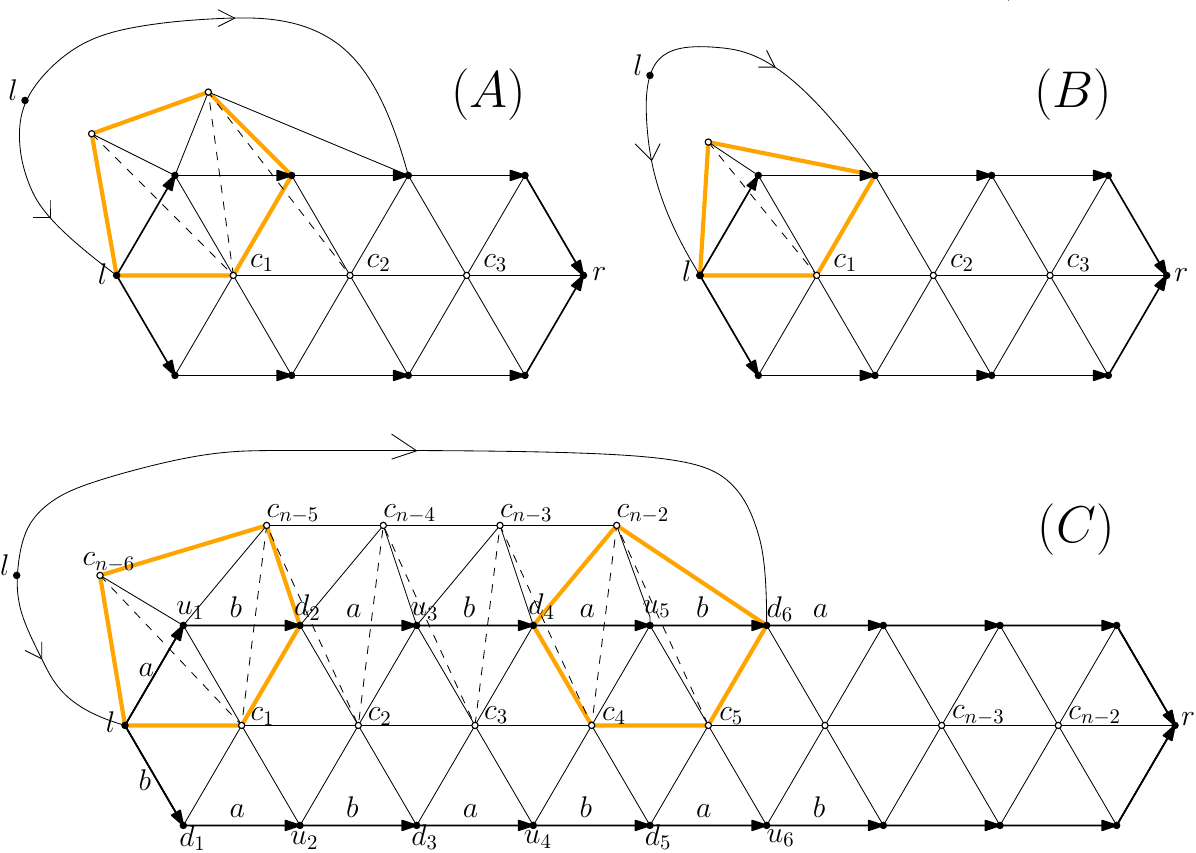}
	\caption{}
	\label{f:cell2}
\end{figure}

\begin{example}
For the $5$-cycle $lc_{n-6}c_{n-5}d_2c_1$ in Figure \ref{f:cell2} (C), there is a unique way to add diagonals between interior vertices, namely we add edges $\overline{c_{n-6}c_1}$ and $\overline{c_{n-5}c_1}$. Similarly, we add edges $\overline{c_{n-2}c_4}$ and $\overline{c_{n-2}c_5}$. However, adding these edges creates new $5$-cycles (e.g.\ $c_1c_{n-5}c_{n-4}u_3c_2$). One can either connect $c_1$ and $c_{n-4}$, or $c_{n-5}$ and $c_2$. We choose the latter and add new edges in a zigzag pattern indicated in Figure \ref{f:cell2} (A) and (C) (see the dashed edges). After adding these edges, we fill in higher dimensional simplexes to obtain a string of $3$-dimensional simplexes, starting from $c_{n-6}u_1c_1l$, and ending at $c_{n-2}d_6c_5u_5$.
\end{example}

Now we give a precise description of the new edges added to $\Xb$. Let $\Lambda$ be the collection of all unordered pairs of cells of $\Xb$ such that their intersection contains at least two edges. Then $DA_n$ acts on $\Lambda$. This action is free. To see this, pick a pair $(\Pi_1,\Pi_2)\in\Lambda$ and suppose $\alpha\in DA_n$ stabilizes it. In particular, $\alpha$ maps $\partial\Pi_1\cap\partial\Pi_2$ to itself. However, $\partial\Pi_1\cap\partial\Pi_2$ is an interval by Corollary \ref{cor:connected intersection}. Thus $\alpha$ fixes a point in $\partial\Pi_1\cap\partial\Pi_2$. Since $DA_n\curvearrowright \Xb$ is free, $\alpha$ is the identity.

Pick a base cell $\Pi$ in $\Xb$ such that $l\in\Pi$ coincides with the identity element of $DA_n$. Let $\Lambda_0$ be the collection of pairs of the form $(\Pi, \ui{i}\Pi)$, $(\Pi, \di{i}\Pi)$ for $i=1,\ldots,n-2$ (here each vertex of $\Pi$ can be identified as an element of $DA_n$, and $\ui{i}\Pi$ means the image of $\Pi$ under the action of $\ui{i}$). Note that
\begin{enumerate}
	\item $\Lambda_0\subset\Lambda$;
	\item different elements in $\Lambda_0$ are in different $DA_n$-orbits;
	\item every $DA_n$-orbit in $\Lambda$ contains an element from $\Lambda_0$.
\end{enumerate}
(1) and (2) follow by direct computation. Pick a pair $(\Pi_1,\Pi_2)\in\Lambda$, by Corollary \ref{cor:connected intersection} (3), one endpoint of $\Pi_1\cap \Pi_2$ is the left tip of $\Pi_1$ or $\Pi_2$, say $\Pi_1$. Let $\alpha\in DA_n$ be the element represented by such left tip. Then $\alpha^{-1}(\Pi_1,\Pi_2)\in\Lambda_0$.

\begin{definition}[Constructing $X$ from $\Xb$]
	\label{def:construct X}
For the pair $(\Pi, \ui{i}\Pi)$, we add an edge between $c_j\in\Pi$ and $\ui{i}c_{j+i}\in\ui{i}\Pi$ for $j=1,\ldots, n-2-i$, and add an edge between $c_j$ and $\ui{i}c_{j+i-1}$ for $j=1,\ldots, n-1-i$, see Figure \ref{fig:2}. For the pair $(\Pi, \di{i}\Pi)$, we add an edge between $c_j$ and $ \di{i}c_{j+i}$ for $j=1,\ldots, n-2-i$, and add an edge between $c_j$ and $\di{i}c_{j+i-1}$ for $j=1,\ldots, n-1-i$, see Figure \ref{fig:1}. Note that the new edges between two cells form a zigzag pattern.

Given a pair of cells $(\Pi_1,\Pi_2)\in \Lambda$, there is a unique element $\alpha\in A$ such that $\alpha^{-1}(\Pi_1,\Pi_2)\in \Lambda_0$. Thus edges between $(\Pi_1,\Pi_2)$ are defined to be the $\alpha$-image of edges between $\alpha^{-1}(\Pi_1,\Pi_2)$. Let $X'$ be the complex obtained by adding all the new edges and let $X$ be the flag completion of $X'$, i.e.\ $X$ is a flag simplicial complex which has the same $1$-skeleton as $X'$. There is a simplicial action $DA_n\curvearrowright X$.
\end{definition}

\begin{figure}[h!]
	\begin{center}
		\begin{tikzpicture}
		\node at (1.5,-2.3) {$\ui{i}c_{i}$};
		\draw [thick] (1.5,-2) -- (3,0);
		\node at (3,0.3) {$c_{1}$};
		\draw [thick] (3,0) -- (3,-2);
		\node at (3,-2.3) {$\ui{i}c_{i+1}$};
		\draw [thick] (3,-2) -- (4.5,0);
		\node at (4.5,0.3) {$c_{2}$};
		\draw [thick] (4.5,0) -- (4.5,-2);
		\node at (4.5,-2.3) {$\ui{i}c_{i+2}$};
		\draw [thick] (4.5,-2) -- (6,0);
		\node at (6,0.3) {$c_{3}$};
		\draw [thick] (6,0) -- (6,-2);
		\node at (6,-2.3) {$\ui{i}c_{i+3}$};
		\draw [dotted, thick] (6.5,-1) -- (7.5,-1);
		\node at (8,0.3) {$c_{n-3-i}$};
		\draw [thick] (8,0) -- (8,-2);
		\node at (8,-2.3) {$\ui{i}c_{n-3}$};
		\draw [thick] (8,-2) -- (9.5,0);
		\node at (9.5,0.3) {$c_{n-2-i}$};
		\draw [thick] (9.5,0) -- (9.5,-2);
		\node at (9.5,-2.3) {$\ui{i}c_{n-2}$};
		\draw [thick] (9.5,-2) -- (11,0);
		\node at (11,0.3) {$c_{n-1-i}$};
		\end{tikzpicture}
	\end{center} 
	\caption{}\label{fig:2}
\end{figure}

\begin{figure}[h!]
	\begin{center}
		\begin{tikzpicture}
		\node at (1.5,2.3) {$\di{i}c_{i}$};
		\draw [thick] (1.5,2) -- (3,0);
		\node at (3,-0.3) {$c_{1}$};
		\draw [thick] (3,0) -- (3,2);
		\node at (3,2.3) {$\di{i}c_{i+1}$};
		\draw [thick] (3,2) -- (4.5,0);
		\node at (4.5,-0.3) {$c_{2}$};
		\draw [thick] (4.5,0) -- (4.5,2);
		\node at (4.5,2.3) {$\di{i}c_{i+2}$};
		\draw [thick] (4.5,2) -- (6,0);
		\node at (6,-0.3) {$c_{3}$};
		\draw [thick] (6,0) -- (6,2);
		\node at (6,2.3) {$\di{i}c_{i+3}$};
		\draw [dotted, thick] (6.5,1) -- (7.5,1);
		\node at (8,-0.3) {$c_{n-3-i}$};
		\draw [thick] (8,0) -- (8,2);
		\node at (8,2.3) {$\di{i}c_{n-3}$};
		\draw [thick] (8,2) -- (9.5,0);
		\node at (9.5,-0.3) {$c_{n-2-i}$};
		\draw [thick] (9.5,0) -- (9.5,2);
		\node at (9.5,2.3) {$\di{i}c_{n-2}$};
		\draw [thick] (9.5,2) -- (11,0);
		\node at (11,-0.3) {$c_{n-1-i}$};
		\end{tikzpicture}
	\end{center} 
	\caption{}\label{fig:1}
\end{figure}

For an interior vertex $v$ in a cell, an edge in the boundary of the cell is \emph{facing} $v$ if 
\begin{itemize}
	\item this edge does not contain a tip;
	\item $v$ and this edge span a triangle in the cell. 
\end{itemize}

Observe that $c_k\in \Pi$ has the edges $\overline{u_kd_{k+1}}$ and $\overline{d_ku_{k+1}}$ facing it for $1\leq k \leq n-2$, see Figure~\ref{f:cell}.

\begin{lemma}
	\label{lem:technical lemma0}
Pick an interior vertex $c_k$ in $\Pi$. Then
\begin{enumerate}
	\item $c_k$ is connected to at least one of the interior vertices of $\ui{i}\Pi$ (resp.\ $\di{i}\Pi$) if and only if at least one of the edges in $\partial \Pi$ facing $c_k$ is contained in $\ui{i}\Pi\cap\Pi$ (resp.\ $\di{i}\Pi\cap\Pi$);
	\item if $c_k\in\Pi$ and $\ui{i}c_{k'}\in \ui{i}\Pi$ (resp.\ $\di{i}c_{k'}\in \di{i}\Pi$) are adjacent, then there is a vertex in $\Pi\cap\partial(\ui{i}\Pi)$ (resp.\ $\Pi\cap\partial(\di{i}\Pi)$) that is adjacent to both $c_k$ and $\ui{i}c_{k'}$ (resp.\ $\di{i}c_{k'}$).
\end{enumerate}
\end{lemma}

\begin{proof}
We only consider the case of $\Pi$ and $\ui{i}\Pi$. Note that $\Pi\cap\partial(\ui{i}\Pi)$ is a path $\omega$ in the lower half of $\Pi$, starting at $l$ and ending at $u_{n-i}$ (if $n-i$ is odd) or $d_{n-i}$ (if $n-i$ is even). Moreover, both $c_k$ and $\ui{i}c_{i+k}$ are facing the $(k+1)$-th edge of $\omega$ for $1\leqslant k\leqslant n-2-i$, $c_{n-1-i}$ is facing the $(n-i)$-th edge of $\omega$, and $\ui{i}c_i$ is facing the
first edge of $\omega$. Now the lemma follows.
\end{proof}

Recall that for two vertices $v_1$ and $v_2$ in a simplicial complex, we write $v_1\sim v_2$ (resp.\ $v_1\nsim v_2$) to denote that they are connected by an edge (resp.\ are not connected by an edge).

\begin{lemma}\
	\label{lem:technical lemma}
	\begin{enumerate}
		\item Suppose $1\leqslant i<j\leqslant n-2$. Then there are exactly two interior vertices in $\ui{j}\Pi$ connected to $\ui{i}c_i$, which are $\ui{j}c_j$ and $\ui{j}c_{j-1}$. There are exactly two interior vertices in $\ui{i}\Pi$ connected to $\ui{j}c_j$, which are $\ui{i}c_i$ and $\ui{i}c_{i+1}$ (see Figure \ref{fig:0} left). Moreover, $\ui{j}c_{j-1}\sim\ui{i}c_{i-1}$ for $i\geqslant 2$.
		\item Suppose $1\leqslant i<n-1$. Then there is exactly one interior vertex of $\ui{n-1}\Pi$ connected to $\ui{i}c_i$, which is $\ui{n-1}c_{n-2}$. Moreover, $\ui{n-1}c_{n-2}\sim\ui{i}c_{i-1}$ for $i\geqslant 2$.
		\item Suppose $1\leqslant i<j\leqslant n-2$. Then there are exactly two interior vertices in $\di{j}\Pi$ connected to $\di{i}c_i$, which are $\di{j}c_j$ and $\di{j}c_{j-1}$. There are exactly two interior vertices in $\di{i}\Pi$ connected to $\di{j}c_j$, which are $\di{i}c_i$ and $\di{i}c_{i+1}$ (see Figure \ref{fig:0} right).  Moreover, $\di{j}c_{j-1}\sim\di{i}c_{i-1}$ for $i\geqslant 2$.
		\item Suppose $1\leqslant i<n-1$. Then there is exactly one interior vertex of $\di{n-1}\Pi$ connected to $\di{i}c_i$, which is $\di{n-1}c_{n-2}$. Moreover, $\di{n-1}c_{n-2}\sim\di{i}c_{i-1}$ for $i\geqslant 2$.
	\end{enumerate}
\end{lemma}
\begin{figure}[h!]
	\begin{center}
		\begin{tikzpicture}
		\node at (1.5,2.3) {$\ui{j}c_{j-1}$};
		\draw [thick] (1.5,2) -- (3,0);
		\node at (3,-0.3) {$\ui{i}c_{i}$};
		\draw [thick] (3,0) -- (3,2);
		\node at (3,2.3) {$\ui{j}c_{j}$};
		\draw [thick] (3,2) -- (4.5,0);
		\node at (4.5,-0.3) {$\ui{i}c_{i+1}$};
		
		\node at (7.5,2.3) {$\di{j}c_{j-1}$};
		\draw [thick] (7.5,2) -- (9,0);
		\node at (9,-0.3) {$\di{i}c_{i}$};
		\draw [thick] (9,0) -- (9,2);
		\node at (9,2.3) {$\di{j}c_{j}$};
		\draw [thick] (9,2) -- (10.5,0);
		\node at (10.5,-0.3) {$\di{i}c_{i+1}$};
		\end{tikzpicture}
	\end{center} 
	\caption{}\label{fig:0}
\end{figure}
\begin{proof}
	We prove (1) and (2), the other items are similar. For (2) we set $j=n-1$. Note that $u_iu^{-1}_j=d^{-1}_{j-i}$ when $u_i$ is in the upper half of the cell, and $u_iu^{-1}_j=u^{-1}_{j-i}$ when $u_i$ is in the lower half of the cell. Now we consider the case where $u_i$ is in the upper half of the cell, the other case is similar. It follows from the scheme of how we add edges between $\Pi$ and $d^{-1}_{j-i}\Pi$ that
	\begin{enumerate}
		\item when $1\leqslant i<j\leqslant n-2$, the only two interior vertices in $\di{j-i}\Pi$ connected to $c_i$ are $d^{-1}_{j-i}c_j$ and $d^{-1}_{j-i}c_{j-1}$; and the only two interior vertices in $\Pi$ connected to $d^{-1}_{j-i}c_j$ are $c_i$ and $c_{i+1}$; moreover, $c_{i-1}\sim d^{-1}_{j-i}c_{j-1}$ for $i\geqslant 2$;
		\item $c_i$ is only adjacent to $d^{-1}_{n-1-i}c_{n-2}$ in the interior of $d^{-1}_{n-1-i}\Pi$; moreover, $c_{i-1}\sim d^{-1}_{n-1-i}c_{n-2}$ for $i\geqslant 2$.
	\end{enumerate}
	Now (1) and (2) follow by applying the action of $\ui{i}$ (note that $\ui{i}d^{-1}_{j-i}=u^{-1}_j$).
\end{proof}

\subsection{Relations to Bestvina's complexes}
We leave a short remark on the relation between $X$ and several complexes defined in Bestvina's paper \cite{Bestvina1999}. We will not prove the statements in this subsection, since they are not used in the later part of the paper, and their proofs follow from the same arguments as in Section~\ref{sec:link of vertices}.

The \emph{central segment} of a cell $\Pi$ in $X$ is the edge path starting at $l$, traveling through $c_1,c_2,\ldots,c_{n-2}$, and arriving at $r$ (see Figure~\ref{f:cell}). A \emph{central line} in $X$ is a subset which is homeomorphic to $\mathbb R$ and is a concatenation of central segments. Note that for each vertex $v\in X$, there is a unique central line in $X$ that contains $v$. Two central lines $\ell_1,\ell_2$ are \emph{adjacent} if there exist vertices $v_1\in\ell_1$ and $v_2\in\ell_2$ such that $v_1$ and $v_2$ are adjacent.

We define a simplicial complex $Z$ as follows. Vertices of $Z$ are in 1-1 correspondence with central lines in $X$. Two vertices are joined by an edge if the corresponding central lines are adjacent. A collection of vertices spans a simplex if each pair of vertices in the collection are joined by an edge. Then $Z$ is isomorphic to the simplicial complex $X(\mathcal G)$ in \cite[Definition 2.3]{Bestvina1999}. Moreover, $X$ is homeomorphic to $Z\times\mathbb R$. There is another complex in Bestvina's paper \cite{Bestvina1999} which is homeomorphic to $Z\times\mathbb R$. It is denoted by $X(\mathcal A)$ and defined in \cite[pp.280]{Bestvina1999}. However, the simplicial structures on $X$ and $X(\mathcal A)$ are different.

\section{Links of vertices in $X$}
\label{sec:link of vertices}
In this section we study local structure of the space $X$ defined in Definition~\ref{def:construct X}.
\subsection{Prisms}
We recall a standard simplicial subdivision of a prism (\cite[Chapter 2.1]{MR1867354}). Let $\Delta^n$ be the $n$-dimensional simplex. Let $P=\Delta^n\times[0,1]$ be a prism. We use $[v_0,\ldots,v_n]$ (resp.\ $[w_0,\ldots.w_n]$) to denote the simplex $\Delta^n\times\{0\}$ (resp.\ $\Delta^n\times\{1\}$). Then $P$ can be subdivided into $(n+1)$-simplexes, each is of the form $[v_0,\ldots,v_i,w_i,\ldots,w_n]$. The prism $P$ with such simplicial structure is called a \emph{subdivided prism}. Note that in the $1$-skeleton $P^{(1)}$, $v_i\sim w_j$ for $j\geqslant i$ and $v_i\nsim w_j$ for $j<i$. This motivates the following definition.

\begin{definition}
	\label{def:prism}
Let $\Gamma$ be a finite simple graph with its vertex set $J$. Suppose there is a partition $J=W\sqcup W'$. We define $\Gamma=\prism{W}{W'}$ if
\begin{enumerate}
	\item $W$ spans a complete subgraph of $\Gamma$, so does $W'$;
	\item $W$ and $W'$ have the same cardinality;
	\item it is possible to order the vertices of $W$ as $\{w_1,\ldots,w_n\}$ such that $W'_1\supsetneq\cdots\supsetneq W'_n\neq \emptyset$, where $W'_i$ is the collection of vertices in $W'$ that are adjacent to $w_i$.
\end{enumerate}
\end{definition}

It is clear from the definition that a simple graph $\Gamma$ is isomorphic to the $1$-skeleton of a subdivided prism if and only if its vertex set has a partition such that $\Gamma=\prism{W}{W'}$.

Definition \ref{def:prism} (2) and (3) imply $w_1$ is connected to each vertex of $W'$, and $w_n$ is connected to only one vertex of $W'$. Moreover, we deduce from (2) and (3) that (3) is also true if we switch the role of $W$ and $W'$. The set $W$ has a linear order, where $w_i<w_j$ if $W'_i\subsetneq W'_j$. Similarly, $W'$ has a linear order. 

Let $\Gamma$ be a graph. Recall that a subgraph $\Gamma'\subset \Gamma$ is a \emph{full subgraph} if $\Gamma'$ satisfies that an edge of $\Gamma$ is inside $\Gamma'$ if and only if the vertices of this edge are inside $\Gamma'$. Let $W\subset\Gamma$ be a collection of vertices. The full subgraph \emph{spanned} by $W$ is the minimal full subgraph that contains $W$.

\begin{definition}
Let $\Gamma'$ be a simplicial graph and let $W$ and $W'$ be two disjoint collections of vertices of $\Gamma'$. We say that $W$ and $W'$ \emph{span a prism} if the full subgraph spanned by $W\cup W'$ is isomorphic to $\prism{W}{W'}$.
\end{definition}

Now we discuss a particular type of graphs which will appear repeatedly in our computation. The reader can proceed directly to Section \ref{subsec: link of a true vertex} and come back when needed.

\begin{definition}
	\label{def:thick hexagon}
A \emph{thick hexagon} is a finite simplicial graph $\Gamma$ such that its vertex set $J$ admits a partition $J=\{c_l\}\sqcup\{c_r\}\sqcup U_l\sqcup U_r\sqcup D_l\sqcup D_r$ satisfying the following conditions: 
\begin{enumerate}
	\item the collection of vertices in $J$ that are adjacent to $c_{l}$ (resp.\ $c_{r}$) is $U_l\cup D_l$ (resp.\ $U_r\cup D_r$);
	\item there are no edges between a vertex in $U_l\cup U_r$ and a vertex in $D_l\cup D_r$;
	\item $U_l$ and $U_r$ span a prism;
	\item $D_l$ and $D_r$ span a prism.
\end{enumerate}
\end{definition}

See Figure \ref{fig:6} for an example of a thick hexagon (edges of the complete subgraphs $U_l,U_r,D_l$ and $D_r$ are not drawn in the picture). Note that if $U_l,U_r,D_l$ and $D_r$ are sets made of a single point, then $\Gamma$ is a $6$-cycle.
\begin{figure}[h!]
	\includegraphics[scale=0.55]{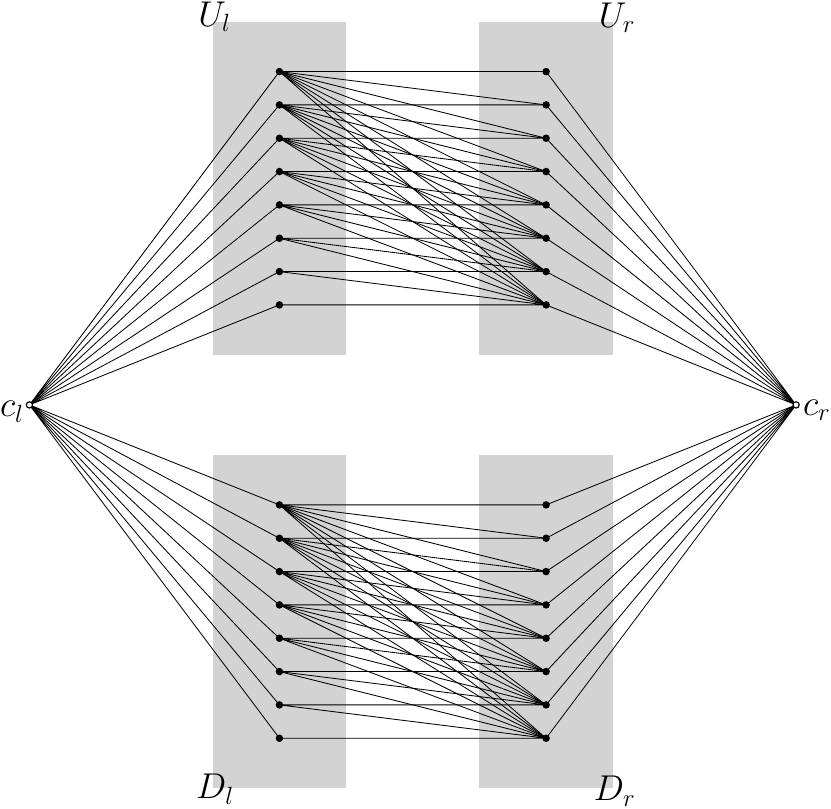}
	\caption{A thick hexagon.}\label{fig:6}
\end{figure}

\begin{lemma}
	\label{lem:thick hexagon is six-large}
Let $\Gamma$ be a thick hexagon. Then $\Gamma$ is $6$-large.
\end{lemma}

\begin{proof}
We use $d$ to denote the combinatorial distance between vertices of $\Gamma$. Note that $d(c_l,c_r)=3$. Let $C$ be a simple $4$-cycle or $5$-cycle in $\Gamma$. We need to show that $C$ has a diagonal. First note that it is impossible that both $c_l$ and $c_r$ are inside $C$, otherwise the length of $C$ is $\geqslant 6$ since it contains two paths from $c_l$ to $c_r$, each of which has length $\geqslant 3$. Since $C$ is simple, we assume without loss of generality that the vertices of $C$ are contained in $c_l\cup U_l\cup U_r$. If $c_l\in C$, then $C$ has a diagonal since each pair of vertices in $U_l$ are connected by an edge. So it remains to consider the case $C\subset\prism{U_l}{U_r}$. 

We assume in addition that there does not exist a pair of non-consecutive vertices in $C$ such that they are both contained in $U_l$ or $U_r$, otherwise $C$ has a diagonal since each of $U_l$ and $U_r$ spans a complete subgraph. It follows from parity considerations that $C$ has to be a $4$-cycle. Let $v_1,v_2,v_3,v_4$ be the consecutive vertices of $C$ such that $v_1,v_2\in U_l$, and $v_3,v_4\in U_r$. If $v_1>v_2$ with respect to the linear order in the above discussion, then $v_1\sim v_3$. If $v_1<v_2$, then $v_2\sim v_4$. Thus in each case $C$ has a diagonal.
\end{proof}

\subsection{Link of a real vertex}
\label{subsec: link of a true vertex}

In this subsection we analyze links of real vertices of the complex $X$ constructed in Definition~\ref{def:construct X}. Since $DA_n$ acts freely and transitively on the set of such vertices it is enough to describe the link of one of them. We pick the real vertex coinciding with
the identity. Following our notation this is the vertex $l\in \Pi$. Otherwise, the same vertex can be described
as: $u_i^{-1}u_i\in u_i^{-1}\Pi$, $d_i^{-1}d_i\in \di{i}\Pi$, or $r^{-1}r\in r^{-1}\Pi$.
%
%
Let $V$ be the set of vertices of $X$ that are adjacent to $l$, and let $\Gamma_V$ be the full subgraph of $X^{(1)}$ spanned by $V$. 
Our goal in this subsection is the following.
\begin{prop}
	\label{prop:real vertex link is six-large}
	The graph $\Gamma_V$ is $6$-large.
\end{prop}
The case $n=2$ is not hard, and we handle it in Remark~\ref{rem:n2} at the end of this subsection. In what follows we assume that $n\geqslant 3$.

Moreover, for reasons that will be explained in Section~\ref{s:general} we need to analyze distances 
in $\Gamma_V$ between some particular vertices. The precise statement is in Lemma~\ref{l:ain far from aout}
below.
 
First, we describe vertices adjacent to $l$ in various copies of $\Pi$. Observe that 
$g^{-1}\Pi$ contains a vertex adjacent to $l$ only for $g\in \{l,r\}\cup \{d_i\}_{i=1}^{n-1}\cup \{u_i\}_{i=1}^{n-1}$. We do not provide proofs of the following three lemmas since they are immediate consequences of 
the form of cells $g^{-1}\Pi$ (see Figure~\ref{f:cell}).
\begin{lemma}\
	\label{l:vertices adj0}
	\begin{enumerate}
		\item There are exactly three vertices in $\Pi$ adjacent to $l$, which are $u_1$, $c_1$, and $d_1$ (see Figure~\ref{fig:L46}).
		\item There are exactly three vertices in $r^{-1}\Pi$ adjacent to $l$, which are $r^{-1}u_{n-1}$, $r^{-1}c_{n-2}$, and $r^{-1}d_{n-1}$ (see Figure~\ref{fig:L46}).		
	\end{enumerate}	
\end{lemma}

\begin{figure}[h!]
	\includegraphics[scale=0.65]{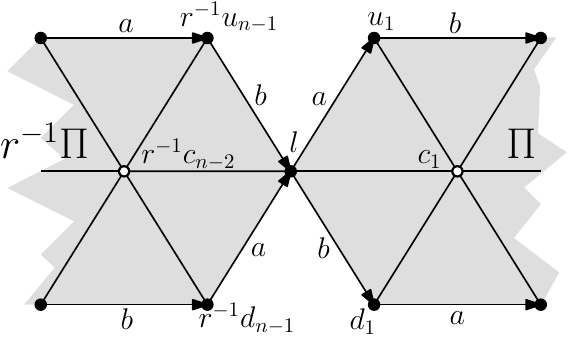}
	\caption{Lemma~\ref{l:vertices adj0}.}
		\label{fig:L46}
\end{figure}

\begin{lemma}\
	\label{l:vertices adj1}
	\begin{enumerate}
		\item There are exactly three vertices in $\di{1}\Pi$ adjacent to $l$, which are $\di{1}l$, $\di{1}c_1$, and $\di{1}u_2$ (see Figure~\ref{fig:L47} (1)).
		\item Suppose that $2\leqslant i \leqslant n-2$. Then there are exactly four vertices in $\di{i}\Pi$
		adjacent to $l$, which are $\di{i}u_{i-1}$, $\di{i}c_{i-1}$, $\di{i}c_{i}$, and $\di{i}u_{i+1}$ (see Figure~\ref{fig:L47} (2)).
		\item There are exactly three vertices in $\di{n-1}\Pi$ adjacent to $l$, which are $\di{n-1}r$, $\di{n-1}c_{n-2}$, and $\di{n-1}u_{n-2}$ (see Figure~\ref{fig:L47} (3)).
	\end{enumerate}	
\end{lemma}

\begin{figure}[h!]
	\includegraphics[scale=0.65]{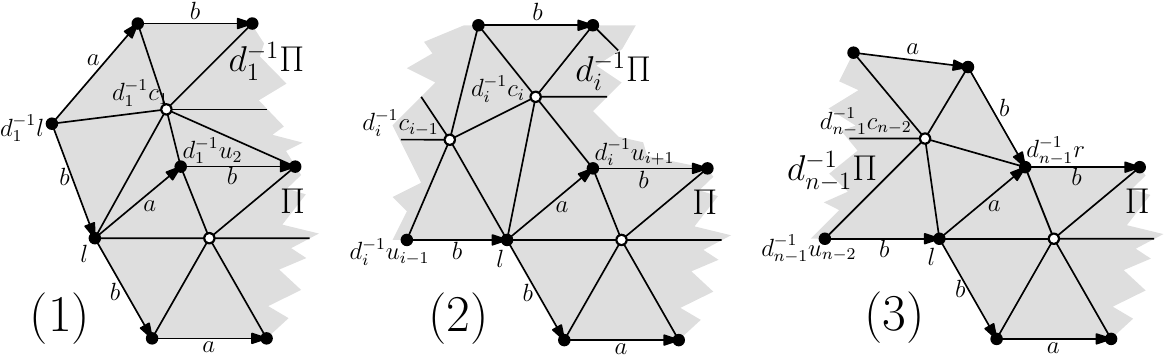}
	\caption{Lemma~\ref{l:vertices adj1}.}
	\label{fig:L47}
\end{figure}

Similarly, Lemma~\ref{l:vertices adj1} still holds with $u$ and $d$ interchanged. That is, we have the following.
\begin{lemma}\
	\label{l:vertices adj2}
	\begin{enumerate}
		\item There are exactly three vertices in $\ui{1}\Pi$ adjacent to $l$, which are $\ui{1}l$, $\ui{1}c_1$, and $\ui{1}d_2$.
		\item Suppose that $2\leqslant i \leqslant n-2$. Then there are exactly four vertices in $\ui{i}\Pi$
		adjacent to $l$, which are $\ui{i}d_{i-1}$, $\ui{i}c_{i-1}$, $\ui{i}c_{i}$, and $\ui{i}d_{i+1}$.
		\item There are exactly three vertices in $\di{n-1}\Pi$ adjacent to $l$, which are $\ui{n-1}r$, $\ui{n-1}c_{n-2}$, and $\ui{n-1}d_{n-2}$.
	\end{enumerate}	
\end{lemma}
Let us make few other immediate observations.
\begin{lemma}
	\label{l:adj real}
	There are exactly four real vertices adjacent to $l$, which are $d_1$, $u_1$, $r^{-1}d_{n-1}$, and
	$r^{-1}u_{n-1}$. The following identifications hold:
	\begin{enumerate}
		\item $\ui{1}l=r^{-1}d_{n-1}$, $\ui{1}d_2=d_1$, $\di{1}l=r^{-1}u_{n-1}$, and $\di{1}u_2=u_1$.
		\item Suppose that $2\leqslant i \leqslant n-2$. Then $\di{i}u_{i-1}=r^{-1}u_{n-1}$,
		$\di{i}u_{i+1}=u_1$, $\ui{i}d_{i-1}=r^{-1}d_{n-1}$, and $\ui{i}d_{i+1}=d_1$.
		\item $\ui{n-1}r=d_1$, $\ui{n-1}d_{n-2}=r^{-1}d_{n-1}$, $\di{n-1}r=u_1$, and $\di{n-1}u_{n-2}=r^{-1}u_{n-1}$.
	\end{enumerate}	
\end{lemma}
We define the following mutually disjoint collections of vertices:
\begin{itemize}
	\item $D_l=\{r^{-1}d_{n-1}\}\cup\{\ui{j}c_{j-1}\}_{j=2}^{n-1}$;
	\item $D_r=\{d_1\}\cup\{\ui{j}c_j\}_{j=1}^{n-2}$;
	\item $U_l=\{r^{-1}u_{n-1}\}\cup\{\di{j}c_{j-1}\}_{j=2}^{n-1}$;
	\item $U_r=\{u_1\}\cup\{\di{j}c_j\}_{j=1}^{n-2}$.
\end{itemize}
By Lemma~\ref{l:vertices adj0}, Lemma~\ref{l:vertices adj1}, and Lemma~\ref{l:vertices adj2}, $V=\{c_{1},r^{-1}c_{n-2}\}\sqcup U_l\sqcup U_r\sqcup D_l\sqcup D_r$. Then Proposition~\ref{prop:real vertex link is six-large} is a consequence of the following result and Lemma \ref{lem:thick hexagon is six-large}.
\begin{figure}[h!]
	\centering
	\includegraphics[width=1\textwidth]{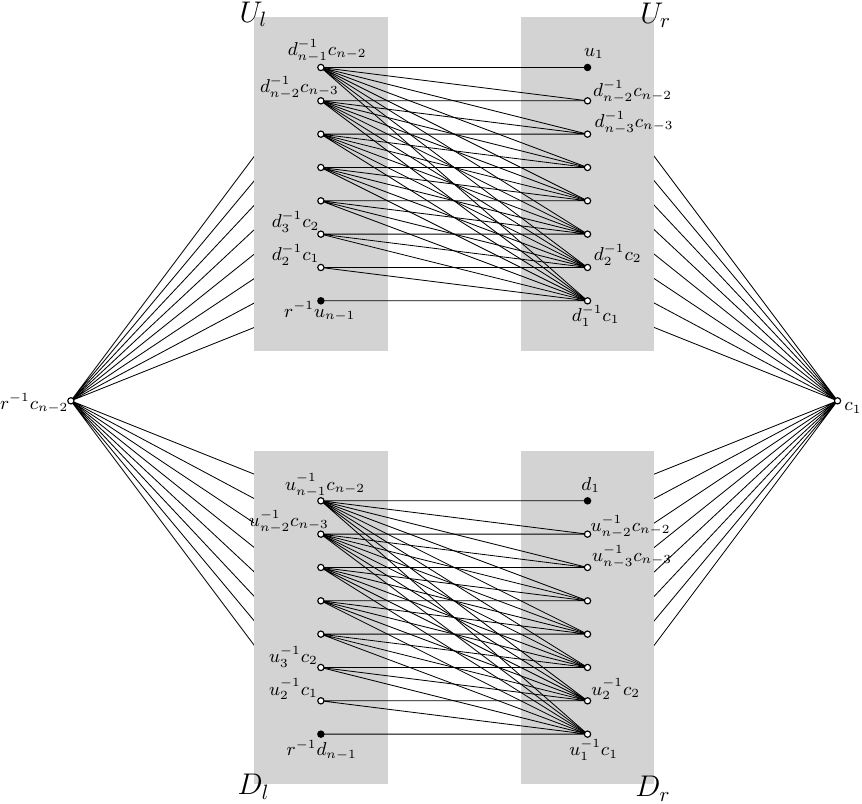}
	\caption{The structure of the link $\Gamma_V$ of a real vertex, the case $n=9$. (Edges in the complete
		graphs $U_r,U_l,D_r$, and $D_l$ are not shown.)}
	\label{f:real_link}
\end{figure}
\begin{prop}
	\label{p:real link is a thick hexagon}
	With the above definition of $U_l,U_r,D_l$ and $D_r$, the graph $\Gamma_V$ satisfies each condition of Definition \ref{def:thick hexagon} with $c_l$ replaced by $r^{-1}c_{n-2}$ and $c_r$ replaced by $c_{1}$ (see Figure~\ref{f:real_link}).
\end{prop}
To prove Proposition~\ref{p:real link is a thick hexagon} we need a few preparatory lemmas.
We will check each item in Definition~\ref{def:thick hexagon}.
\begin{lemma}\
	\label{l:prism1}
	\begin{enumerate}
		\item $c_1\nsim r^{-1}c_{n-2}$, and $c_1\nsim r^{-1}d_{n-1}, r^{-1}u_{n-1}$, and 
		$r^{-1}c_{n-2}\nsim u_1,d_1$.
		\item $c_1\sim d_1, u_1$, and $r^{-1}c_{n-2}\sim r^{-1}d_{n-1},r^{-1}u_{n-1}$.
		\item For $1\leqslant j \leqslant n-2$, we have $c_1 \sim \ui{j}c_j,\di{j}c_j$, and
		$c_1 \nsim \ui{j+1}c_j,\di{j+1}c_j$.
		\item For $1\leqslant j \leqslant n-2$, we have $r^{-1}c_{n-2} \sim \ui{j+1}c_j,\di{j+1}c_j$, and
		$r^{-1}c_{n-2} \nsim \ui{j}c_j,\di{j}c_j$.

	\end{enumerate}
	Consequently, the collection of vertices in $V$ that are adjacent to $c_1$ (resp.\ $r^{-1}c_{n-2}$) is 
	$U_r\sqcup D_r$ (resp.\ $U_l\sqcup D_l$).
\end{lemma}
\begin{proof}
	Observe that $\Pi \cap r^{-1}\Pi = \{l \}$, thus by Lemma~\ref{lem:technical lemma0}, there is no edge between $c_1$ and $r^{-1}c_{n-2}$. 
	Furthermore, it follows that $r^{-1}d_{n-1},r^{-1}u_{n-1}$ are real vertices not belonging to $\Pi$, hence $c_1\nsim r^{-1}d_{n-1}, r^{-1}u_{n-1}$. Similarly, $r^{-1}c_{n-2}\nsim u_1,d_1$. This shows (1).	
	Edges in (2) are edges in cells $\Pi$ and $r^{-1}\Pi$.
	Recall that the only vertices in $\ui{j}\Pi$ adjacent to $c_1\in \Pi$ are 
	$\ui{j}c_j$, and $\ui{j}c_{j+1}$ (see Definition~\ref{def:construct X}, as well as Figure~\ref{fig:2} with all the $i$'s in the figure replaced by $j$), and the only vertices in $\di{j}\Pi$ adjacent to $c_1$ are $\di{j}c_j$, and $\di{j}c_{j+1}$ (see Figure~\ref{fig:1}), hence (3).
	For (4) the argument is similar: the only vertices in $u_j\Pi$ adjacent to $c_{n-2}$ are 
	$u_{j}c_{n-2-j}$, and $u_jc_{n-1-j}$, and the only vertices in $d_j\Pi$ adjacent to $c_{n-2}$ are
	$d_jc_{n-2-j}$, and $d_jc_{n-1-j}$. Thus (4) follows by using the translation invariance of Definition~\ref{def:construct X} and translating the information in the previous sentence by $r^{-1}$. Note that $r^{-1}u_j$ is equal to $\ui{n-j}$ (for $(n-j)$ even) or $\di{n-j}$ (for $(n-j)$ odd), and that $r^{-1}d_j$ is equal to $\di{n-j}$ (for $(n-j)$ even) or $\ui{n-j}$ (for $(n-j)$ odd).
\end{proof}

\begin{lemma}\
	\label{l:prism2}
	\begin{enumerate}
		\item No two of the vertices $d_1, r^{-1}d_{n-1}, u_1, r^{-1}u_{n-1}$ are adjacent.
		\item For $1\leqslant j \leqslant n-2$, we have $u_1 \nsim \ui{j}c_j,\ui{j+1}c_j$, and
		$r^{-1}u_{n-1} \nsim \ui{j}c_j,\ui{j+1}c_j$.
		\item For $1\leqslant j \leqslant n-2$, we have $d_1 \nsim \di{j}c_j,\di{j+1}c_j$, and
		$r^{-1}d_{n-1} \nsim \di{j}c_j,\di{j+1}c_j$.
		\item For $1\leqslant j,k \leqslant n-2$, we have $\ui{j}c_j \nsim \di{k}c_k,\di{k+1}c_k$, and
		$\ui{j+1}c_j \nsim \di{k}c_k,\di{k+1}c_k$.
	\end{enumerate}
		Consequently, no vertex in $U_l\cup U_r$ is adjacent to a vertex in $D_l\cup D_r$.
\end{lemma}
\begin{proof}
		All the vertices in (1) are real, and there are no edges between them in $\Xb$. Therefore,
		(1) follows from the fact that passing from $\Xb$ to $X$ we do not add edges between real vertices.
		For (2), observe that $u_1 \notin \ui{j}\Pi$ so that $u_1 \nsim \ui{j}c_j,\ui{j+1}c_j$. 
		By Lemma~\ref{l:adj real} (1), we have $u_1^{-1}l=r^{-1}d_{n-1}$.
		For $j\geqslant 2$, by Lemma~\ref{l:adj real} (2), we have $u_j^{-1}d_{j-1}=r^{-1}d_{n-1}$.
		Therefore, for $1\leqslant j \leqslant n-2$, we have $r^{-1}d_{n-1}\in u_j^{-1}\Pi$. Since, 
		by Corollary~\ref{cor:connected intersection} (2), the intersection of two cells is contained either
		in the upper half or in the lower half, we have that $r^{-1}u_{n-1}\notin u_j^{-1}\Pi$.		
		(3) can be proven the same way as (2), interchanging $u$ and $d$. For (4), observe that $\ui{j}\Pi\cap\Pi$ is contained in the lower 
		half of $\partial \Pi$, and $\di{j}\Pi\cap\Pi$ is contained in the upper 
		half of $\partial \Pi$. Then Corollary~\ref{cor:disjoint} implies that $\di{j}\Pi \cap \ui{k}\Pi$ is at most one point, hence no internal vertices of the two cells are adjacent by Lemma~\ref{lem:technical lemma0} (1).
\end{proof}
\begin{lemma}
	\label{l:prism3}
	Every two vertices in $D_r$ are adjacent. The same is true for $U_l,U_r$, and $D_l$. 
\end{lemma}
\begin{proof}
	We prove the statement for $D_r$ and $D_l$. The other cases are proven similarly.

	By Lemma~\ref{l:adj real}(1), we have $d_1=\ui{1}d_2$. Hence, by the form of the cell $\ui{1}\Pi$
	we have $d_1\sim \ui{1}c_1$ (see Figure~\ref{f:cell}). Similarly, by Lemma~\ref{l:adj real}(2),
	we have $d_1=\ui{j}d_{j+1}$, hence $d_1\sim \ui{j}c_j$, for $2\leqslant j\leqslant n-2$.
	By Lemma~\ref{lem:technical lemma}(1), we have that $\ui{j}c_j\sim \ui{k}c_k$, for $1\leqslant j < k \leqslant n-2$. Therefore, every two vertices in $D_r$ are adjacent.  

	By Lemma~\ref{l:adj real} (2)(3), we have $r^{-1}d_{n-1}=\ui{j}d_{j-1}$, for $2\leqslant j\leqslant n-1$.
		Hence, by the form of the cell $\ui{j}\Pi$
	we have $\ui{j}d_{j-1}\sim \ui{j}c_{j-1}$ (see Figure~\ref{f:cell}). 
	By Lemma~\ref{lem:technical lemma}(1)(2), we have 
	$\ui{j}c_{j-1}\sim \ui{k}c_{k-1}$, for $2\leqslant j<k\leqslant n-1$. Therefore, every two vertices in $D_l$ are adjacent.  	
\end{proof}
Now we show that $D_l$ and $D_r$ span a prism. This relies on the following lemmas.
\begin{lemma}
	\label{l:prism4}
	The only vertex in $D_l$ (resp.\ $D_r$) adjacent to $d_1$ (resp.\ $r^{-1}d_{n-1}$) is
	$\ui{n-1}c_{n-2}$ (resp.\ $\ui{1}c_1$).
\end{lemma}
\begin{proof}
	By Lemma~\ref{l:prism2}(1), we have $d_1\nsim r^{-1}d_{n-1}$. For $2\leqslant j \leqslant n-2$, we have $d_1=\ui{j}d_{j+1}$ (by Lemma~\ref{l:adj real}(2)), hence $d_1\nsim \ui{j}c_{j-1}$.
	Since $d_1=\ui{n-1}r$ (by Lemma~\ref{l:adj real}(3)), we have $d_1\sim \ui{n-2}c_{n-2}$.
	It follows that $\ui{n-2}c_{n-2}$ is the only vertex in $D_l$ adjacent to $d_1$.
	The statement about $r^{-1}d_{n-1}$ and $D_r$ has an analogous proof.
\end{proof}
\begin{lemma}
	\label{l:prism5}
	Let $1\leqslant j,k \leqslant n-2$. Then $\ui{j}c_j \sim \ui{k+1}c_k$ 
	iff $j\leqslant k+1$.
\end{lemma}
\begin{proof}
	If $j=k+1$ then $\ui{j}c_j \sim \ui{k+1}c_k$
	since $c_{j-1}\sim c_{j}$ in $\Pi$. 
	If $j<k+1$ and $k+1\leqslant n-2$ then $u_j^{-1}c_j\sim u_{k+1}^{-1}c_k$, by 
	Lemma~\ref{lem:technical lemma} (1).
	If $j<k+1$ and $k+1= n-1$ then $u_j^{-1}c_j\sim u_{k+1}^{-1}c_k$, by 
	Lemma~\ref{lem:technical lemma} (2).
	If $j>k+1$ then $u_j^{-1}c_j\nsim u_{k+1}^{-1}c_k$, by 
	Lemma~\ref{lem:technical lemma} (1).
\end{proof}
\medskip

\noindent
\emph{Proof of Proposition~\ref{p:real link is a thick hexagon}.}
By Lemma~\ref{l:prism3}, graphs spanned by, respectively, $U_r,U_l,D_r$, and $D_l$ are complete, hence the condition (3) in Definition~\ref{def:thick hexagon} is satisfied for $\Gamma_V$. 
Now, we prove that the condition (4) from Definition~\ref{def:thick hexagon} holds, that is, $D_r$ and
$D_l$ span a prism. Clearly, the cardinalities of the two sets are equal to $n-1$.
We have to order appropriately (see Definition~\ref{def:prism} (3)) vertices of $D_r$
as $\{w_1,\ldots, w_{n-1}\}$.
We set $w_{n-1}=d_1$, and $w_j=\ui{j}c_{j}$, for $1\leqslant j\leqslant n-2$.
Define $W_{1}':=U_l$, and $W_j':=\{\ui{k}c_{k-1}\}_{k=2}^{j}$, for $2\leqslant j \leqslant n-1$.    
Then, by Lemma~\ref{l:prism4} and Lemma~\ref{l:prism5}, the set $W_j'$ is exactly the collection of vertices  
in $D_l$ that are adjacent to $w_j$. It follows that $D_r\cup D_l$ spans a prism $\prism{D_l}{D_r}$. 
Analogously one proves that $U_r\cup U_l$ spans a prism $\prism{U_l}{U_r}$, hence the condition (5) from
Definition~\ref{def:thick hexagon} holds. The properties (1) and (2) from the same definition 
hold by, respectively, Lemma~\ref{l:prism1} and Lemma~\ref{l:prism2}.
\hfill$\square$

Now we record several observations which will be used later in Section~\ref{s:general} when we glue the complexes for dihedral Artin groups together.

\begin{lemma}\
	\label{l:ain far from aout}
 We have
 \begin{align*}
 	d_{\Gamma_V}(r^{-1}u_{n-1},r^{-1}d_{n-1})=d_{\Gamma_V}(u_{1},d_1)=\\d_{\Gamma_V}(r^{-1}u_{n-1},u_1)=d_{\Gamma_V}(r^{-1}u_{n-1},d_1)=2,
 \end{align*} 
 and 
 \begin{align*}
 	d_{\Gamma_V}(r^{-1}d_{n-1},u_1)=d_{\Gamma_V}(r^{-1}u_{n-1},d_1)=3.
 \end{align*}

\end{lemma}
\begin{proof}
	The first statement follows by Lemma~\ref{l:prism1} (2) and Lemma~\ref{l:prism2} (1).
	It is clear that $d_{\Gamma_V}(r^{-1}d_{n-1},u_1),d_{\Gamma_V}(r^{-1}u_{n-1},d_1)\leqslant 3$.
	By Lemma~\ref{l:prism2} (1), we have $d_{\Gamma_V}(r^{-1}d_{n-1},u_1)\geqslant 2$, and $d_{\Gamma_V}(r^{-1}u_{n-1},d_1)\geqslant 2$. If we had $d_{\Gamma_V}(r^{-1}d_{n-1},u_1)=2$ then there
	would exist a vertex adjacent to both $r^{-1}d_{n-1}$, and $u_1$. This would contradict Proposition~\ref{p:real link is a thick hexagon}. Similarly one shows that $d_{\Gamma_V}(r^{-1}u_{n-1},d_1)=3$.
\end{proof}
\begin{remark}
	\label{rem:n2}	
	\begin{figure}[h]
		\centering
		\includegraphics[width=0.8\textwidth]{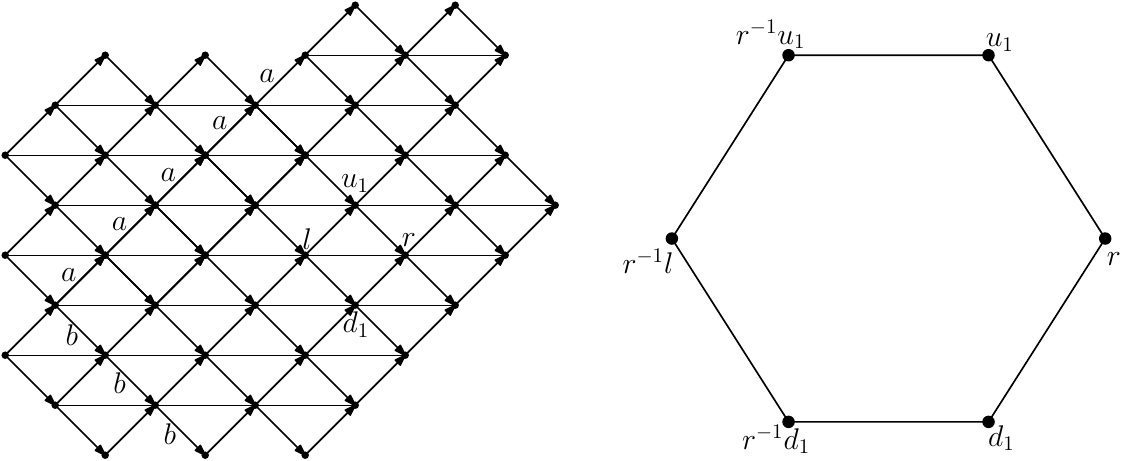}
		\caption{The complex $X$ in the case $n=2$ (a fragment on the left), and the link of its vertex (right).}
		\label{f:link_n2}
	\end{figure}
	The group $DA_2$ is isomorphic to $\mathbb Z^2$. To obtain the complex $X$ we add the diagonal 
	$\overline{lr}$ in every precell being the square (we do not add interior vertices $c_i$). 
	As a result, $X$ is isomorphic to the equilateral triangulation
	of the Euclidean plane, and links of vertices are $6$-cycles; see Figure~{\ref{f:link_n2}}. Note that
	it is a special case of a thick hexagon (see Definition~\ref{def:thick hexagon}), where all $U_l,U_r,D_l$, and $D_r$ reduce to single vertices,
	and hence the prisms $\prism{U_l}{U_r}$ etc.\ reduce to edges.
\end{remark}

Let $\Xa$ be as in Section~\ref{subsec:precells} with its edges labeled by the two generators $a$ and $b$ of $DA_n$. Recall that $\Gamma_V$ is the link of the identity vertex $l$ in $X^{(1)}$, and the real vertices of $X^{(1)}$ comes from the vertices of $\Xa$. Thus there are exactly four real vertices in $\Gamma_V$, two of them arise from incoming and outgoing $a$-edges at $l$, which we denoted by $a^+$ and $a^-$; and two of them arise from incoming and outgoing $b$-edges at $l$, which we denoted by $b^+$ and $b^-$. On the other hand, by Lemma~\ref{l:adj real}, the four real vertices in $\Gamma_V$ are described as $d_1,u_1,r^{-1}d_{n-1}$ and $r^{-1}u_{n-1}$. One readily verifies the identification $u_1=a^-$, $d_1=b^-$, $r^{-1}u_{n-1}=b^+$ and $r^{-1}d_{n-1}=a^+$ (see Figure~\ref{f:cell2} (C)). Thus we have the following result, where the first item follows from Lemma~\ref{l:ain far from aout} and the second item follows from Remark~\ref{rem:n2}
\begin{lemma}\
		\label{lem:distance}
\begin{enumerate}
	\item If $n\geqslant 3$, then $d_{\Gamma_V}(a^+,b^+)=d_{\Gamma_V}(a^+,b^-)=d_{\Gamma_V}(a^-,b^+)=d_{\Gamma_V}(a^-,b^-)=2$ and $d_{\Gamma_V}(a^+,a^-)=d_{\Gamma_V}(b^+,b^-)=3$;
	\item if $n=2$, then $d_{\Gamma_V}(a^+,b^+)=d_{\Gamma_V}(a^-,b^-)=2$, $d_{\Gamma_V}(a^+,b^-)=d_{\Gamma_V}(a^-,b^+)=1$ and $d_{\Gamma_V}(a^+,a^-)=d_{\Gamma_V}(b^+,b^-)=3$.
\end{enumerate}
\end{lemma}

\subsection{Link of an interior vertex}
Pick an interior vertex $c_i\in \Pi$ and we will fix $c_i$ for the rest of this subsection. Let $V$ be the set of vertices of $X$ that are adjacent to $c_i$, and let $\Gamma_V$ be the full subgraph of $X^{(1)}$ spanned by $V$. Our goal in this subsection is the following.
\begin{prop}
	\label{prop:interior vertex link is six-large}
The graph $\Gamma_V$ is $6$-large.
\end{prop}

First we characterize elements of $V$. They fall into two disjoint classes:
\begin{enumerate}[label=\Alph*]
	\item vertices in $\Pi$ that are adjacent to $c_i$;
	\item vertices outside $\Pi$ that are adjacent to $c_i$, they must be interior vertices of some cells other than $\Pi$.
\end{enumerate}
Class A consists of six vertices around $c_i$. There are two cases.
\begin{enumerate}
	\item The number $i$ is odd. Then $c_i$ is connected to $u_i$ and $d_{i+1}$ in the upper half of $\partial\Pi$, and $d_i$ and $u_{i+1}$ in the lower half of $\partial\Pi$. In such case $c_i$ is facing a $b$-edge in the upper half and an $a$-edge in the lower half. Moreover, $c_i$ is connected to $c_{i-1}$ or $l$ (when $i=1$) on the left, and $c_{i+1}$ or $r$ (when $i=n-2$) on the right.
	\item The number $i$ is even. Then $c_i$ is connected to $d_i$ and $u_{i+1}$ in the upper half of $\partial\Pi$, and $u_i$ and $d_{i+1}$ in the lower half of $\partial\Pi$. In such case $c_i$ is facing an $a$-edge in the upper half and a $b$-edge in the lower half. Moreover, $c_i$ is connected to $c_{i-1}$ on the left, and $c_{i+1}$ or $r$ (when $i=n-2$) on the right.
\end{enumerate}

Now we assume that $i$ is odd. The case of even $i$ is similar. Next we study vertices of class B. Let $w\Pi$ be a cell such that it contains interior vertices that are adjacent to $c_i$. By Lemma \ref{lem:technical lemma0} (1), $w\Pi\cap \Pi$ is a path of length $\geqslant 2$ such that it contains either $\overline{u_id_{i+1}}$ or $\overline{d_iu_{i+1}}$. There are two cases.
\begin{enumerate}
	\item If $\overline{u_id_{i+1}}\subset w\Pi\cap \Pi$, then $w\Pi\cap \Pi$ contains vertex $u_i$ and the $b$-edge emanating from $u_i$, therefore $w^{-1}u_i$ must be a vertex $u_j$ for some $j\neq i$ or $w^{-1}u_i=l$. Then $w=u_i\ui{j}$ for $0\leqslant j\leqslant n-1$ and $j\neq i$ (we set $u_0=l$). We define $p_j=u_i\ui{j}$.
	\item If $\overline{d_iu_{i+1}}\subset w\Pi\cap \Pi$, then similar to the previous case, we deduce that $w=d_i\di{j}$ for $0\leqslant j\leqslant n-1$ and $j\neq i$ (set $d_0=l$). We define $q_j=d_i\di{j}$.
\end{enumerate}

\begin{lemma}\
	\label{lem:vertices adjacent1}
For a fixed $c_i$, the following hold.
\begin{enumerate}
	\item There is only one vertex in $p_0\Pi$ adjacent to $c_i$, which is $p_0c_1$.
	\item Suppose $1\leqslant j<i\leqslant n-2$. Then there are exactly two vertices in $p_j\Pi$ adjacent to $c_i$, which are $p_jc_j$ and $p_jc_{j+1}$.
	\item Suppose $1\leqslant i<j\leqslant n-2$. Then there are exactly two vertices in $p_j\Pi$ adjacent to $c_i$, which are $p_jc_{j-1}$ and $p_jc_{j}$.
	\item There is only one vertex in $p_{n-1}\Pi$ adjacent to $c_i$, which is $p_{n-1}c_{n-2}$.
\end{enumerate}
\end{lemma}

\begin{proof}
Recall that the only vertex in $\Pi$ adjacent to $\ui{i}c_i\in \ui{i}\Pi$ is $c_1$. Thus (1) follows by applying the action of $u_i$.
For (2), we apply Lemma \ref{lem:technical lemma} (1) with $i$ and $j$ interchanged to deduce that $\ui{i}c_i$ is adjacent to $\ui{j}c_j$ and $\ui{j}c_{j+1}$. Then (2) follows by applying the action of $u_i$ (recall that $p_j=u_i\ui{j}$). Assertion (3) follows from Lemma \ref{lem:technical lemma} (1) in a similar way, and (4) follows from Lemma \ref{lem:technical lemma} (2).
\end{proof}

The following lemma can be proved in a similar way to Lemma \ref{lem:vertices adjacent1}, using Lemma \ref{lem:technical lemma} (3) and (4).
\begin{lemma}
	 \label{lem:vertices adjacent2}
Lemma \ref{lem:vertices adjacent1} still holds with $p$ replaced by $q$.
\end{lemma}

We define the following mutually disjoint collections of vertices:
\begin{itemize}
	\item $U_l=\{u_i\}\cup\{p_jc_j\}_{j=1}^{i-1}\cup\{p_jc_{j-1}\}_{j=i+1}^{n-1}$;
	\item $U_r=\{d_{i+1}\}\cup\{p_jc_{j+1}\}_{j=0}^{i-1}\cup\{p_jc_j\}_{j=i+1}^{n-2}$;
	\item $D_l=\{d_i\}\cup\{q_jc_j\}_{j=1}^{i-1}\cup\{q_jc_{j-1}\}_{j=i+1}^{n-1}$;
	\item $D_r=\{u_{i+1}\}\cup\{q_jc_{j+1}\}_{j=0}^{i-1}\cup\{q_jc_j\}_{j=i+1}^{n-2}$;
\end{itemize}
By Lemma \ref{lem:vertices adjacent1} and Lemma \ref{lem:vertices adjacent2}, $V=\{c_{i-1},c_{i+1}\}\cup U_l\cup U_r\cup D_l\cup D_r$ (when $i=1$, let $c_{i-1}=l$, when $i=n-2$, let $c_{i+1}=r$). Then Proposition \ref{prop:interior vertex link is six-large} is a consequence of the following result and Lemma \ref{lem:thick hexagon is six-large}.
\begin{figure}[h!]
	\centering
	\includegraphics[width=1\textwidth]{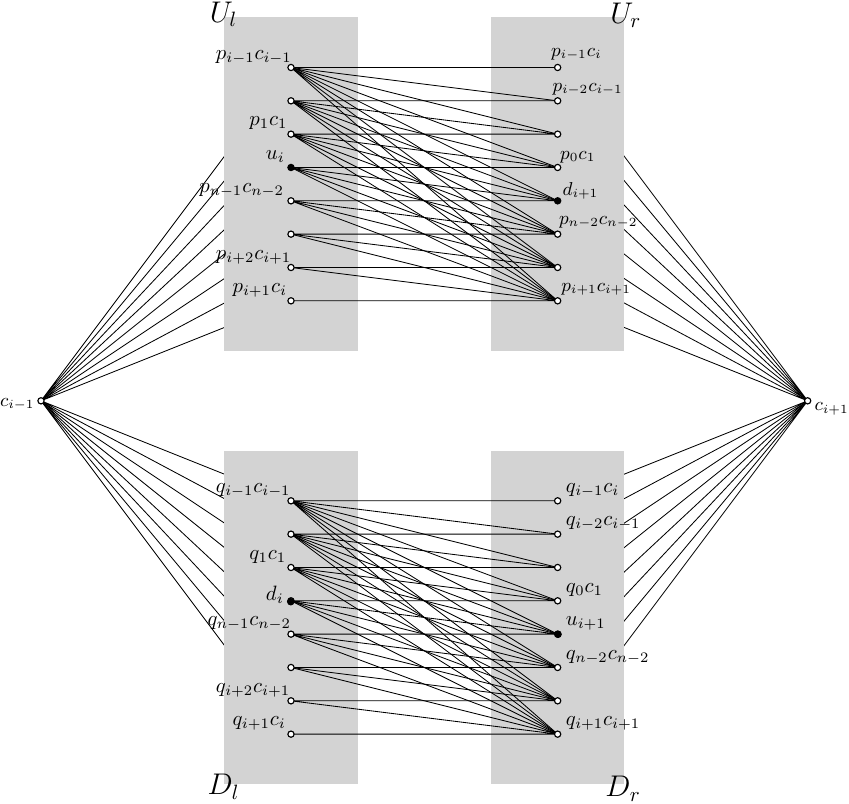}
	\caption{The structure of the link $\Gamma_V$ of the interior vertex $c_i$, the case $n=9$. (Edges in the complete
		graphs $U_r,U_l,D_r$, and $D_l$ are not shown.)}
	\label{f:interior_link}
\end{figure}

\begin{prop}
	\label{prop:interior link is a thick hexagon}
With the above definition of $U_l,U_r,D_l$ and $D_r$, the graph $\Gamma_V$ satisfies each condition of Definition \ref{def:thick hexagon} with $c_l$ replaced by $c_{i-1}$ and $c_r$ replaced by $c_{i+1}$.
\end{prop}

The rest of this section is devoted to the proof of Proposition \ref{prop:interior link is a thick hexagon}.

\begin{lemma} Set $c_{n-1}=r$ and $c_0=l$.
\begin{enumerate}
	\item We have $c_{i-1}\nsim p_0c_1$ and $c_{i+1}\sim p_0c_1$.
	\item For $1\leqslant j<i\leqslant n-2$, $c_{i-1}\sim p_jc_j$, $c_{i-1}\nsim p_jc_{j+1}$, $c_{i+1}\nsim p_jc_j$, and $c_{i+1}\sim p_jc_{j+1}$.
	\item For $1\leqslant i<j\leqslant n-2$, $c_{i-1}\sim p_jc_{j-1}$, $c_{i-1}\nsim p_jc_{j}$, $c_{i+1}\nsim p_jc_{j-1}$, and $c_{i+1}\sim p_jc_{j}$.
	\item We have $c_{i-1}\sim p_{n-1}c_{n-2}$ and $c_{i+1}\nsim p_{n-1}c_{n-2}$.
\end{enumerate}
Moreover, all the statements still hold with $p$ replaced by $q$. As a consequence, the collection of vertices in $V$ that are adjacent to $c_{i-1}$ (resp.\ $c_{i+1}$) is $U_l\cup D_l$ (resp.\ $U_r\cup D_r$).
\end{lemma}

\begin{proof}
For (1), it follows from the fact that the zigzag pattern between $\ui{i}\Pi$ and $\Pi$ is as in Figure \ref{fig:2} that $c_1\sim u^{-1}_ic_{i+1}$ and $c_1\nsim \ui{i}c_{i-1}$. Then (1) follows by applying the action of $u_i$. Now we prove (2). If $i<n-2$, then we apply Lemma \ref{lem:technical lemma} (1) with $i$ and $j$ interchanged to deduce that $\ui{i}c_i\sim \ui{j}c_j$ and $\ui{i}c_i\sim \ui{j}c_{j+1}$. By the zigzag pattern (see Definition~\ref{def:construct X} and Figure \ref{fig:3} left), we know $\ui{i}c_{i-1}\sim \ui{j}c_j$, $\ui{i}c_{i-1}\nsim \ui{j}c_{j+1}$, $\ui{i}c_{i+1}\nsim\ui{j}c_j$ and $\ui{i}c_{i+1}\sim\ui{j}c_{j+1}$. Thus (2) follows by applying the action of $u_i$. If $i=n-2$, then $u_ju^{-1}_{n-2}r=u_{j+2}$. Thus $u^{-1}_{n-2}r=\ui{j}u_{j+2}$ is connected to $\ui{j}c_{j+1}$. Then we have a similar zigzag pattern as in Figure \ref{fig:3} right. Assertion (3) is similar to (2) (since $i<j$, we have a zigzag pattern as in Figure \ref{fig:4}), and (4) follows from Lemma~\ref{lem:technical lemma} (2).
\begin{figure}[ht!]
	\begin{center}
		\begin{tikzpicture}
		\node at (1.5,-0.3) {$\ui{i}c_{i-1}$};
		\draw [thick] (1.5,0) -- (3,2);
		\node at (3,2.3) {$\ui{j}c_j$};
		\draw [thick] (3,2) -- (3,0);
		\node at (3,-0.3) {$\ui{i}c_{i}$};
		\draw [thick] (3,0) -- (4.5,2);
		\node at (4.5,2.3) {$\ui{j}c_{j+1}$};
		\draw [thick] (4.5,2) -- (4.5,0);
		\node at (4.5,-0.3) {$\ui{i}c_{i+1}$};
		\node at (7.2,-0.3) {$\ui{n-2}c_{n-3}$};
		\draw [thick] (7.5,0) -- (9,2);
		\node at (9,2.3) {$\ui{j}c_j$};
		\draw [thick] (9,2) -- (9,0);
		\node at (9,-0.3) {$\ui{n-2}c_{n-2}$};
		\draw [thick] (9,0) -- (10.5,2);
		\node at (10.5,2.3) {$\ui{j}c_{j+1}$};
		\draw [thick] (10.5,2) -- (10.5,0);
		\node at (10.5,-0.3) {$\ui{n-2}r$};
		\end{tikzpicture}
	\end{center} 
	\caption{}\label{fig:3}
\end{figure}
\begin{figure}[ht!]
	\begin{center}
		\begin{tikzpicture}
		\node at (1.5,-0.3) {$\ui{i}c_{i-1}$};
		\draw [thick] (1.5,0) -- (1.5,2);
		\node at (1.5,2.3) {$\ui{j}c_{j-1}$};
		\draw [thick] (1.5,2) -- (3,0);
		\node at (3,-0.3) {$\ui{i}c_{i}$};
		\draw [thick] (3,0) -- (3,2);
		\node at (3,2.3) {$\ui{j}c_{j}$};
		\draw [thick] (3,2) -- (4.5,0);
		\node at (4.5,-0.3) {$\ui{i}c_{i+1}$};
		\node at (7.5,-0.3) {$\ui{1}l$};
		\draw [thick] (7.5,0) -- (7.5,2);
		\node at (7.5,2.3) {$\ui{j}c_{j-1}$};
		\draw [thick] (7.5,2) -- (9,0);
		\node at (9,-0.3) {$\ui{1}c_{1}$};
		\draw [thick] (9,0) -- (9,2);
		\node at (9,2.3) {$\ui{j}c_{j}$};
		\draw [thick] (9,2) -- (10.5,0);
		\node at (10.5,-0.3) {$\ui{1}c_{2}$};
		\end{tikzpicture}
	\end{center} 
	\caption{}\label{fig:4}
\end{figure}
\end{proof}

\begin{lemma}
Let $v\in U_l\cup U_r$ and $v'\in D_l\cup D_r$. Then $v$ and $v'$ are not adjacent.
\end{lemma}

\begin{proof}
We first look at the case when $v$ and $v'$ are interior vertices. Suppose $v\in w\Pi$ and $v'\in w'\Pi$. Recall that $\overline{u_id_{i+1}}\subset w\Pi$ and $\overline{d_iu_{i+1}}\subset w'\Pi$. Thus $w\Pi\cap\Pi$ (resp.\ $w'\Pi\cap\Pi$) is contained in the upper (resp.\ lower) half of $\partial\Pi$ by Corollary \ref{cor:connected intersection} (2). Then Corollary \ref{cor:disjoint} implies $w\Pi\cap w'\Pi$ is at most one point, thus $v$ and $v'$ are not adjacent.

The case where neither of $v$ and $v'$ is interior is clear. It remains to consider the case when only one of $v$ and $v'$, say $v'$, is interior. Each real vertex adjacent to $v'$ is inside $\partial(w'\Pi)$. However, $w'\Pi\cap \Pi$ is contained in the lower half of $\partial\Pi$, thus $v$ and $v'$ are not adjacent.
\end{proof}

\begin{lemma}
Every two vertices in $U_l$ are connected by an edge. The same is true for $U_r,D_l$ and $D_r$.
\end{lemma}

\begin{proof}
We only prove $U_l$ spans a complete subgraph, and $U_r$ spans a complete subgraph. The cases for $D_l$ and $D_r$ are similar. First we prove $u_i$ is connected to other vertices in $U_l$. Note that $p_ju_j=u_i$. However, $u_j\sim c_{j-1}$ and $u_j\sim c_j$. By applying the action of $p_j$ and using the invariance in Definition~\ref{def:construct X}, we have $u_i=p_ju_j\sim p_j c_{j-1}$ and $u_i=p_ju_j\sim p_j c_j$. Similarly, by using $p_jd_{j+1}=d_{i+1}$ and applying the action of $p_j$ to $d_{j+1}\sim c_j$ and $d_{j+1}\sim c_{j+1}$,  we know that $d_{i+1}$ is connected to other vertices in $U_r$.

By Lemma \ref{lem:technical lemma} (1), $\ui{j}c_{j}\sim \ui{j'}c_{j'}$ for $1\leqslant j,j'\leqslant n-2$. By applying the action of $u_i$, each of $\{p_jc_j\}_{j=1}^{i-1}$ and $\{p_jc_j\}_{j=i+1}^{n-2}$ spans a complete subgraph. Moreover, we deduce from Lemma \ref{lem:technical lemma} (1) together with the zigzag pattern (cf. Definition~\ref{def:construct X}) between $\ui{j}\Pi$ and $\ui{j'}\Pi$ (see Figure~\ref{fig:3} and Figure~\ref{fig:4}) that $\ui{j}c_{j-1}\sim \ui{j'}c_{j'-1}$ for $2\leqslant j,j'\leqslant n-2$ and $\ui{j}c_{j+1}\sim \ui{j'}c_{j'+1}$ for $1\leqslant j,j'\leqslant n-3$. By Lemma \ref{lem:technical lemma} (2), we know actually $\ui{j}c_{j-1}\sim \ui{j'}c_{j'-1}$ for $2\leqslant j,j'\leqslant n-1$. By Definition~\ref{def:construct X} (see Figure \ref{fig:2}), $c_1\sim \ui{j'}c_{j'+1}$ for $1\leqslant j'\leqslant n-3$. Applying the action of $u_i$, each of $\{p_jc_{j-1}\}_{j=i+1}^{n-1}$ and $\{p_jc_{j+1}\}_{j=0}^{i-1}$ spans a complete subgraph.

By Lemma \ref{lem:technical lemma} (1) and (2), $\ui{j}c_{j}\sim \ui{j'}c_{j'-1}$ for $1\leqslant j<j'\leqslant n-1$ and $\ui{j}c_{j+1}\sim \ui{j'}c_{j'}$ for $1\leqslant j<j'\leqslant n-2$. Moreover, $c_1\sim \ui{j'}c_{j'}$ for $1\leqslant j'\leqslant n-2$ (see Figure \ref{fig:2}). By applying the action of $u_i$, we know that a vertex from $\{p_jc_j\}_{j=1}^{i-1}$ (resp.\ $\{p_jc_{j+1}\}_{j=0}^{i-1}$) and a vertex from $\{p_jc_{j-1}\}_{j=i+1}^{n-1}$ (resp.\ $\{p_jc_{j}\}_{j=i+1}^{n-2}$) are adjacent. Now it follows that each of $U_l$ and $U_r$ spans a complete subgraph.
\end{proof}

Now we show $U_l$ and $U_r$ span a prism. This relies on the following four lemmas.
\begin{lemma}
	\label{lem:prism1}
Suppose $i<j\leqslant n-2$. Then a vertex $v\in U_l$ is adjacent to $p_jc_j$ if and only if $v\in \{u_i\}\cup \{p_kc_k\}_{k=1}^{i-1}\cup \{p_kc_{k-1}\}_{k=j}^{n-1}$.
\end{lemma}

\begin{proof}
Note that $p_ju_j=u_i$. Since $c_j\sim u_j$, $u_i=p_ju_j\sim p_jc_j$. Suppose $1\leqslant k\leqslant i-1$. Thus $1\leqslant k<j\leqslant n-2$. By Lemma \ref{lem:technical lemma} (1), $\ui{k}c_k\sim\ui{j}c_j$, thus $p_kc_k\sim p_jc_j$. Suppose $i+1\leqslant k<j$. Then $\ui{k}c_{k-1}\nsim\ui{j}c_j$ by Lemma \ref{lem:technical lemma} (1), thus $p_{k}c_{k-1}\nsim p_{j}c_j$. It is clear that $p_jc_{j-1}\sim p_jc_j$. Suppose $j<k\leqslant n-1$. Then $\ui{j}c_j\sim \ui{k}c_{k-1}$ by Lemma \ref{lem:technical lemma} (1) and (2), hence $p_{j}c_{j}\sim p_{k}c_{k-1}$.
\end{proof}

\begin{lemma}
	\label{lem:prism2}
A vertex $v\in U_l$ is adjacent to $d_{i+1}\in U_r$ if and only if $v\in \{u_i\}\cup \{p_kc_k\}_{k=1}^{i-1}\cup \{p_{n-1}c_{n-2}\}$.
\end{lemma}

\begin{proof}
It is clear that $u_i\sim d_{i+1}$. Since $u_i=p_ku_k$, $d_{i+1}=p_kd_{k+1}$. Suppose $1\leqslant k\leqslant n-2$. Since $d_{k+1}\sim c_k$ and $d_{k+1}\nsim c_{k-1}$, $d_{i+1}=p_kd_{k+1}\sim p_kc_k$ and $d_{i+1}\nsim p_kc_{k-1}$. Suppose $k=n-1$. Then $d_{i+1}=p_{n-1}r$. Since $r\sim c_{n-2}$, $p_{n-1}c_{n-2}\sim p_{n-1}r=d_{i+1}$.
\end{proof}

\begin{lemma}
	\label{lem:prism3}
	A vertex $v\in U_l$ is adjacent to $p_0c_1\in U_r$ if and only if $v\in \{u_i\}\cup \{p_kc_k\}_{k=1}^{i-1}$.
\end{lemma}

\begin{proof}
Recall that $p_0=u_i$. Since $c_1\sim l$, $p_0c_1\sim u_il=u_i$. Suppose $1\leqslant k\leqslant n-2$. Then $c_1\sim \ui{k}c_k$, and $c_1\nsim \ui{k}c_{k-1}$ when $k>1$. By applying the action of $p_0=u_i$, $p_0c_1$ is adjacent to each vertex in $\{p_kc_k\}_{k=1}^{i-1}$ when $\{p_kc_k\}_{k=1}^{i-1}$ is nonempty, and $p_0c_1$ is adjacent to none of $\{p_kc_{k-1}\}_{k=i+1}^{n-2}$. When $k=n-1$, $\Pi$ and $\ui{n-1}\Pi$ only intersect along one edge, then $p_0\Pi$ and $p_{n-1}\Pi$ do as well. Thus interior vertices of $p_0\Pi$ and interior vertices of $p_{n-1}\Pi$ are not adjacent, in particular $p_0c_1\nsim p_{n-1}c_{n-2}$.
\end{proof}

\begin{lemma}
	\label{lem:prism4}
Suppose $1\leqslant j\leqslant i-1$. A vertex $v\in U_l$ is adjacent to $p_jc_{j+1}\in U_r$ if and only if $v\in \{p_kc_k\}_{k=j}^{i-1}$.
\end{lemma}

\begin{proof}
Since $c_{j+1}\nsim u_j$, we have $\ui{j}c_{j+1}\nsim \ui{j}u_j=l$. Hence $p_{j}c_{j+1}\nsim u_il=u_i$. Suppose $1\leqslant k<j$. Then $\ui{k}c_k\nsim\ui{j}c_{j+1}$ by Lemma \ref{lem:technical lemma} (1). Thus $p_kc_k\nsim p_jc_{j+1}$. Suppose $k=j$. Then it is clear that $p_jc_j\sim p_jc_{j+1}$. Suppose $j<k\leqslant n-2$. Then $\ui{k}c_k\sim\ui{j}c_{j+1}$ and $\ui{k}c_{k-1}\nsim\ui{j}c_{j+1}$ by Lemma \ref{lem:technical lemma} (1). Thus $p_kc_k\sim p_jc_{j+1}$ and $p_kc_{k-1}\nsim p_jc_{j+1}$. It follows that $p_jc_{j+1}$ is adjacent to each of $\{p_kc_k\}_{k=j+1}^{i-1}$ and is adjacent to none of $\{p_kc_{k-1}\}_{k=i+1}^{n-2}$. Suppose $k=n-1$. By Lemma \ref{lem:technical lemma} (2), $u^{-1}_{n-1}c_{n-2}\sim \ui{j}c_j$ and $u^{-1}_{n-1}c_{n-2}\sim \ui{j}c_{j-1}$. Thus $u^{-1}_{n-1}c_{n-2}\nsim \ui{j}c_{j+1}$ and $p_{n-1}c_{n-2}\nsim p_jc_{j+1}$.
\end{proof}

It follows from Lemma \ref{lem:prism1}, Lemma \ref{lem:prism2}, Lemma \ref{lem:prism3} and Lemma \ref{lem:prism4} that $U_l$ and $U_r$ span a prism with the linear order on $U_r$ given by $p_{i+1}c_{i+1}>\ldots>p_{n-2}c_{n-2}>d_{i+1}>p_0c_1>\ldots>p_{i-1}c_i$. Similarly, we can prove $D_l$ and $D_r$ span a prism. Thus all the conditions in Definition \ref{def:thick hexagon} are satisfied and we have finished the proof of Proposition \ref{prop:interior link is a thick hexagon}.

\section{The complexes for Artin groups of almost large type}
\label{s:general}
Let $A_\Gamma$ be an Artin group with defining graph $\Gamma$. Let $\Gamma'\subset\Gamma$ be a full subgraph with induced edge labeling and let $A_{\Gamma'}$ be the Artin group with defining graph $\Gamma'$. The following is proved in \cite{Van1983homotopy}.
\begin{theorem}
	\label{thm:lek}
	Let $\Gamma_1$ and $\Gamma_2$ be full subgraphs of $\Gamma$ with the induced edge labelings. Then
	\begin{enumerate}
		\item the natural homomorphism $A_{\Gamma_1}\to A_\Gamma$ is injective;
		\item $A_{\Gamma_1}\cap A_{\Gamma_2}=A_{\Gamma_1\cap\Gamma_2}$.
	\end{enumerate} 
\end{theorem}

Subgroups of $A_{\Gamma}$ of the form $A_{\Gamma'}$ are called \emph{standard subgroups}.

Let $P_{\Gamma}$ be the standard presentation complex of $A_{\Gamma}$, and let $X^{\ast}_{\Gamma}$ be the universal cover of $P_{\Gamma}$. We orient each edge in $P_{\Gamma}$ and label each edge in $P_{\Gamma}$ by a generator of $A_\Gamma$. Thus edges of $X^{\ast}_{\Gamma}$ have induced orientation and labeling. There is a natural embedding $P_{\Gamma'}\hookrightarrow P_\Gamma$. Since $A_{\Gamma'}\to A_{\Gamma}$ is injective, $P_{\Gamma'}\hookrightarrow P_\Gamma$ lifts to various embeddings $X^{\ast}_{\Gamma'}\to X^{\ast}_{\Gamma}$. Subcomplexes of $X^{\ast}_{\Gamma}$ arising in such way are called \emph{standard subcomplexes}. The subgraph $\Gamma'$ is the \emph{defining graph} of this standard subcomplex.

Now we assume $A_{\Gamma}$ is of almost large type.
Recall that it means that in the defining graph $\Gamma$ there is no triangle with an edge labeled by two and no square with three edges labeled by two.
A \emph{block} of $X^{\ast}_{\Gamma}$ is a standard subcomplex which comes from an edge in $\Gamma$. This edge is called the \emph{defining edge} of the block. The block is \emph{large} (resp.\ \emph{small}) if its defining edge is labeled by an integer $\geqslant 3$ (resp.\ $=2$). 

The following is a direct consequence of Theorem~\ref{thm:lek}.
\begin{corollary}
	\label{cor:intersection of blocks}
Let $B_1$ and $B_2$ be blocks of $X^\ast_\Gamma$ with defining edges $e_1$ and $e_2$, respectively. Suppose $x$ is a vertex in $B_1\cap B_2$. Let $E$ be the standard subcomplex of $X^\ast_\Gamma$ containing $x$ with defining graph $e_1\cap e_2$ (note that $E=\{x\}$ when $e_1\cap e_2=\emptyset$). Then $B^{(0)}_1\cap B^{(0)}_2=E^{(0)}$.
\end{corollary}

We define precells of $X^{\ast}_{\Gamma}$ as in Section \ref{subsec:precells}, and subdivide each precell as in Figure \ref{f:cell} to obtain a simplicial complex $\Xb_\Gamma$. Interior vertices and real vertices of $\Xb_{\Gamma}$ are defined in a similar way. 

\begin{definition}[Constructing $X_\Gamma$]
	\label{def:construct XGa}
Within each block of $\Xb_\Gamma$, we add edges between interior vertices as in Definition~\ref{def:construct X}. Since each element of $A_\Gamma$ maps one block to another block with the same defining edge, and the stabilizer of each block is a conjugate of a standard subgroup of $A_{\Gamma}$, one readily verifies that the newly added edges are compatible with the action of deck transformations $A_{\Gamma}\curvearrowright\Xb_\Gamma$. Let $X'_\Gamma$ be the complex obtained by adding all the new edges, and let $X_{\Gamma}$ be the flag completion of $X'_\Gamma$. The action $A_{\Gamma}\curvearrowright\Xb_\Gamma$ extends to a simplicial action $A_{\Gamma}\curvearrowright X_{\Gamma}$, which is proper and cocompact. A \emph{block} in $X_{\Gamma}$ is defined to be the full subcomplex spanned by vertices in a block of $\Xb_\Gamma$.
\end{definition}

\begin{lemma}
	\label{lem:at most one edge}
If two cells are in different blocks of $\Xb_\Gamma$, then their intersection is at most one edge.
\end{lemma}

\begin{proof}
Let $B_1$ and $B_2$ be two different blocks in $\Xa_\Gamma$ and let $C_i\subset B_i$ be precells for $i=1,2$. It suffices to show $C_1\cap C_2$ is connected. By considering the quotient homomorphism from $A_\Gamma$ to its associated Coxeter group, we know that the inclusion of $1$-skeleta $C^{(1)}_i\hookrightarrow (\Xa_\Gamma)^{(1)}$ is isometric with respect to the path metric. As $B^{(1)}$ is convex with respect to the path metric on $(\Xa_\Gamma)^{(1)}$ (\cite{MR3291260}), $C_1\cap C_2$ is connected.
\end{proof}

Lemma~\ref{lem:at most one edge} is the reason why we did not add edges between interior vertices from different blocks in Definition~\ref{def:construct XGa}.

\begin{lemma}
	\label{lem:iso}
	The isomorphism between a block in $X^{\ast}_{\Gamma}$ and the space $\Xa$ in Section \ref{subsec:precells} naturally extends to an isomorphism between a block in $X_{\Gamma}$ and the space $X$ in Section \ref{subsec:subividing and adding new edges}.
\end{lemma}

\begin{proof}
	By our construction, it suffices to show that if two vertices $v_1$ and $v_2$ in a block $B\subset X^{\ast}_\Gamma$ are not adjacent in this block, then they are not adjacent in $\Xa_\Gamma$. However, this follows from the fact that $B^{(1)}$ is convex with respect to the path metric on the $1$-skeleton of $X^{\ast}_\Gamma$ (\cite{MR3291260}).
\end{proof}

\begin{lemma}
	\label{lem:sc}
The complex $X_\Gamma$ is simply connected.
\end{lemma}

\begin{proof}
Let $f$ be an edge of $X_{\Gamma}$ not in $\Xb_\Gamma$. Since $\partial f$ is inside one block, we assume without loss of generality that $f$ connects an interior point of $\Pi$ and an interior point of $\ui{i}\Pi$. Lemma \ref{lem:technical lemma0} (2) implies that $f$ and a vertex in $\Pi\cap\ui{i}\Pi$ span a triangle. By flagness of $X_{\Gamma}$, $f$ is homotopic rel its end points to the concatenation of other two sides of this triangle, which is inside $\Xb_\Gamma$. 

Now we show that each loop in $X_{\Gamma}$ is null-homotopic. It suffices to consider the case where this loop is a concatenation of edges of $X_{\Gamma}$. If some edges of this loop are not in $\Xb_\Gamma$, then we can find homotopies from these edges rel their end points to paths in $\Xb_\Gamma$ by the previous discussion. Thus this loop is homotopic to a loop in $\Xb_{\Gamma}$, which must be null-homotopic since $\Xb_\Gamma$ is simply connected. 
\end{proof}

\begin{lemma}
	\label{lem:link six-large}
The link of each vertex in the $1$-skeleton $X^{(1)}_{\Gamma}$ is a $6$-large graph.
\end{lemma}

\begin{proof}
Let $x\in X^{(1)}_{\Gamma}$ be a vertex. If $x$ is an interior vertex, then there is a unique block $B\subset X_{\Gamma}$ containing this vertex, and any other vertex in $X^{(1)}_{\Gamma}$ adjacent to $x$ is contained in this block. Since $B$ is a full subcomplex of $X_{\Gamma}$, we have lk$(x,X^{(1)}_{\Gamma})=$lk$(x,B^{(1)})$. The latter link is $6$-large by Lemma~\ref{lem:iso} and Proposition~\ref{prop:interior vertex link is six-large}.

Let $x$ be a real vertex and let $\omega$ be a simple $4$-cycle or $5$-cycle in lk$(x,X^{(1)}_\Gamma)$. We need to show that $\omega$ has a diagonal. Define a vertex $v\in \omega$ to be \emph{special} if the edge $\overline{xv}$ is inside $\Xa_\Gamma$. Note that special vertices are real, but the converse may not be true (in a small block every vertex is real, yet there are edges not in $\Xa_\Gamma$). 

First we consider the case when the number of special vertices in $\omega$ is $\leqslant 1$. We claim $\omega$ is contained in one block $B$. If the contrary holds, then $\omega$ contains at least two vertices which are in the intersection of two different blocks that contain $x$. However, these two vertices have to be special as the vertex set of the intersection of two different blocks containing $x$ is determined by Corollary~\ref{cor:intersection of blocks}. This yields a contradiction. Note that lk$(x,B)$ is $6$-large by Proposition~\ref{prop:real vertex link is six-large}. Thus $\omega$ has a diagonal. 

Now we assume that $\omega$ has $\geqslant 2$ special vertices. Let $\{v_i\}_{i\in \mathbb Z/n\mathbb Z}$ be consecutive special vertices on $\omega$ (then $n$ is the number of special vertices on $\omega$). By the argument in the previous paragraph, the segment $\overline{v_iv_{i+1}}$ of $\omega$ is contained in one block, which we denote by $B_i$. Then $\overline{v_iv_{i+1}}$ is an edge-path in lk$(x,B_i)$ traveling between two different special vertices (since $n\geqslant 2$). Thus if $B_i$ is large, then $\overline{v_iv_{i+1}}$ has length $\geqslant 2$ by Lemma~\ref{lem:distance}. Therefore the number of large blocks among $\{B_i\}_{i=1}^{n}$ is $\leqslant 2$.

Each $v_i$ arises from an edge between $x$ and $v_i$. This edge is inside $\Xa_\Gamma$, hence it is labeled by a generator of $A_\Gamma$, corresponding to a vertex $z_i\in \Gamma$. Since $v_i$ corresponds to either an incoming, or an outgoing edge labeled by $z_i$, we will also write $v_i=z^+_i$ or $v_i=z^-_i$. Let $e_i$ be the defining edge of $B_i$. Then 
\begin{equation}
\label{eq:intersection}
z_{i+1}\in e_i\cap e_{i+1}.
\end{equation} 
Moreover, note that lk$(x,B_i)$ is a circle when $B_i$ is a small block. Thus 
\begin{equation}
\label{eq:intersection1}
z_i\neq z_{i+1}\ \mathrm{when}\ B_i\ \mathrm{is\ a\ small\ block.}
\end{equation} 
However, it is possible that $z_i=z_{i+1}$ when $B_i$ is large.
\medskip

\noindent
\emph{Case 1:} There are two large blocks. Note that there is at most one small block, so the two large blocks must be consecutive. We assume without loss of generality that $B_0$ and $B_1$ are large. We claim $B_0=B_1$, and there are no other blocks. Then $\omega$ is inside one block and by Proposition~\ref{prop:real vertex link is six-large}, it has a diagonal. Now we prove the claim. 

First we show there are no other blocks. By contradiction we assume there is a small block $B_2$. If $B_0\neq B_1$, then $e_0$, $e_1$ and $e_2$ are pairwise distinct. We deduce from \eqref{eq:intersection} that $z_1,z_2$ and $z_0$ form a triangle in $\Gamma$ which contains an edge labeled by $2$, contradiction. 
If $B_0=B_1$ then \eqref{eq:intersection1} 
implies that $z_0\neq z_2$, and hence either $z_2 = z_1$ or
$z_0 = z_1$. Suppose that $z_2=z_1$. Then, by Lemma~\ref{lem:distance} we have that the lengths
of the paths $\overline{v_1v_2}$,  $\overline{v_2v_0}$, and $\overline{v_0v_1}$ are at least, respectively,
$3,1$, and $2$. This implies that $\omega$ has length $\geqslant 6$, which is a contradiction. Similarly, we get
a contradiction for $z_0 = z_1$.

%
%

Now we show $B_0=B_1$. If $B_0\neq B_1$, since there are no other blocks, we must have $z_0=z_1$ by \eqref{eq:intersection}. Now we argue as before to show that the length of $\omega$ is $\geqslant 6$, which is a contradiction.
\medskip

\noindent
\emph{Case 2:} There is only one large block. We denote this block by $B_0$ and claim $n=1$. If there are other small blocks, then $n\leqslant 4$. 

We first show that $n=4$ is impossible. We argue by contradiction. Let $B_1,B_2$ and $B_3$ be small blocks. If these small blocks are pairwise distinct, then \eqref{eq:intersection} and \eqref{eq:intersection1} imply that 
either $z_0=z_1$ and we have a triangle in $\Gamma$ with all labels $2$ or
$z_0,z_1,z_2$ and $z_3$ are consecutive vertices in a $4$-cycle of $\Gamma$ with three edges labeled by $2$. In both cases we get a contradiction. If $B_1=B_2$, then the by observation \eqref{eq:intersection1} above, the concatenation of $\overline{v_1v_2}$ and $\overline{v_2v_3}$ has length $=3$. As $\overline{v_3v_0}$ has length $\geqslant 1$ and $\overline{v_0v_1}$ has length $\geqslant 2$, we know that $\omega$ has length $\geqslant 6$, which yields a contradiction. Then case $B_3=B_2$ can be ruled out similarly. 
If $B_1 = B_3$ and $z_0\neq z_1$ then, by \eqref{eq:intersection1}, $z_3=z_1$ and hence $z_0,z_1$ are in a small block. This contradicts
the fact that $z_0,z_1\in B_0$. Hence, $z_0 = z_1$.
Thus the segment $\overline{v_0v_1}$ is an edge path in $B_0$ from $z^{+}_0$ to $z^{-}_0$, which has length $\geqslant 3$ by Lemma~\ref{lem:distance}. On the other hand, the concatenation of $\overline{v_1v_2}$, $\overline{v_2v_3}$ and $\overline{v_3v_0}$ has length $\geqslant 3$, hence $\omega$ has length $\geqslant 6$, which is a contradiction. 

Now we consider $n=3$. Let $B_1,B_2$ be small blocks. If $B_1\neq B_2$, then \eqref{eq:intersection} and \eqref{eq:intersection1} imply that $z_0,z_1$ and $z_2$ form three vertices of a $3$-cycle in $\Gamma$ with two edges labeled by $2$, which is a contradiction. If $B_1=B_2$, then \eqref{eq:intersection1} implies that $z_0=z_1$. Thus the segment $\overline{v_0v_1}$ is an edge path in $B_0$ from $z^{+}_0$ to $z^{-}_0$, which has length $\geqslant 3$ by Lemma~\ref{lem:distance}. On the other hand, \eqref{eq:intersection1} implies the concatenation of $\overline{v_1v_2}$ and $\overline{v_2v_0}$ has length $\geqslant 3$. Thus $\omega$ has length $\geqslant 6$, which is a contradiction.

It follows from \eqref{eq:intersection1} that the case $n=2$ is impossible.
\medskip

\noindent
\emph{Case 3}: there are no large blocks. Then $n\leqslant 5$. Since each $\overline{v_iv_{i+1}}$ has length $\leqslant 2$, we have $n\geqslant 3$. If $n=3$, then \eqref{eq:intersection} and \eqref{eq:intersection1} imply that $z_0,z_1$ and $z_2$ form three vertices of a $3$-cycle in $\Gamma$ with all edges labeled by $2$, which is a contradiction. 

Now suppose $n=4$. If all blocks are pairwise distinct, then by \eqref{eq:intersection} and \eqref{eq:intersection1}, we have a $4$-cycle in $\Gamma$ with all edges labeled by $2$, which is a contradiction. If two consecutive blocks, say $B_0$ and $B_1$, are the same, then \eqref{eq:intersection1} implies that concatenation of $\overline{v_0v_1}$ and $\overline{v_1v_2}$ has length $=3$ and $z_0=z_2$. 
It follows that $B_2=B_3$ and the length of $\omega$ is $6$, a contradiction.
%
If two non-consecutive blocks, say $B_0$ and $B_2$, are the same then, by \eqref{eq:intersection1}, 
we have $z_0=z_2$, hence $B_0=B_1$, which we ruled out above. 

It remains to consider $n=5$. If two consecutive blocks are the same, then	 by the argument in the previous paragraph, we know $\omega$ has length $\geqslant 6$, which is impossible. If two non-consecutive blocks are the same, then we can deduce a contradiction as before. Thus the blocks are pairwise distinct. Then \eqref{eq:intersection} implies that the defining edges of these blocks form a $5$-cycle in $\Gamma$. Since the length of $\omega$ is $\leqslant 5$, $v_i$ and $v_{i+1}$ are adjacent for all $i$. Suppose without loss of generality that $v_0=z^{+}_0$, then we must have $v_1=z^{-}_1$ by Lemma~\ref{lem:distance} (2). Similarly, $v_2=z^+_2$, $v_3=z^-_3$, $v_4=z^+_4$ and $v_0=z^-_0$, which is a contradiction.
\end{proof}

The following is a direct consequence of Lemma~\ref{lem:sc} and Lemma~\ref{lem:link six-large}.

\begin{theorem}
	\label{thm:main}
The complex $X_{\Gamma}$ is systolic. Hence if $A_{\Gamma}$ is an Artin group of almost large type, then it acts properly and cocompactly by automorphisms on a systolic complex $X_{\Gamma}$.
\end{theorem}

\bibliography{mybib}{}
\bibliographystyle{plain}
\end{document}